\documentclass[12pt,reqno]{amsart}
\usepackage{amsmath,amssymb,cite,geometry,bm,easybmat,enumerate}
\usepackage[colorlinks,citecolor=blue,linkcolor=blue,urlcolor=blue]{hyperref}
\newtheorem{theorem}{Theorem}[section]
\newtheorem{lemma}[theorem]{Lemma}
\newtheorem{proposition}[theorem]{Proposition}
\newtheorem{corollary}[theorem]{Corollary}

\newtheorem{remark}[theorem]{Remark}
\numberwithin{equation}{section}
\allowdisplaybreaks[4]
\date{\today}

\begin{document}

\title[Spectral invariants of the Dirichlet-to-Neumann map]{Spectral invariants of the Dirichlet-to-Neumann map associated to the Witten-Laplacian with potential}

\author{Xiaoming Tan}
\address{Beijing International Center for Mathematical Research, Peking University, Beijing 100871, China}
\email{tanxm@pku.edu.cn}


\keywords{Spectral invariants; Heat trace; Asymptotic expansion; Dirichlet-to-Neumann map; Steklov eigenvalues.\\
{\bf-----------------}\\
\hspace*{3mm} {\it Email address}: tanxm@pku.edu.cn\\
\hspace*{3mm} Beijing International Center for Mathematical Research, Peking University, Beijing 100871, China}

\begin{abstract}
    For a compact connected Riemannian manifold with smooth boundary, we establish an effective procedure, by which we can calculate all the coefficients of the spectral asymptotic formula of the Dirichlet-to-Neumann map associated to the Witten-Laplacian with potential. In particular, by computing the full symbol of the Dirichlet-to-Neumann map we explicitly give the first four coefficients. They are spectral invariants, which provide precise information concerning the volume, curvatures, drifting function and potential.
\end{abstract}

\maketitle 

\section{Introduction}

Let $(M,g)$ be a compact connected Riemannian manifold of dimension $n$ with smooth boundary $\partial M$. Given a smooth real-valued function $\phi \in C^{\infty}(M)$, which is called the drifting function. The triple $(M,g,e^{-\phi}dV)$ is customarily called a smooth metric measure space or manifolds with density, where $dV$ is the Riemannian volume measure on $M$. For any $u \in C^{\infty}(M)$, the Witten-Laplacian $\Delta_{\phi}$ (also called drift Laplacian, drifted Laplacian, drifting Laplacian, $\phi$-Laplacian, weighted Laplacian, or Bakry--\'{E}mery Laplacian) is defined by (see, for example, \cite{FutaLL13,Ara22,ChengZhou17})
\begin{align}\label{1.1}
    \Delta_{\phi}u:=\Delta u - g(\nabla\phi,\nabla u),
\end{align}
where $\Delta$ and $\nabla$ denote the Laplace--Beltrami operator and the gradient operator on the manifold, respectively.

It is well-known that the Witten-Laplacian $\Delta_{\phi}$ is a densely defined self-adjoint operator in the space of square-integrable functions on $M$ with respect to the measure $e^{-\phi}dV$, that is, for $u,v\in C^{\infty}_{0}(M)$, the following integration by parts formula holds (see \cite{FutaLL13,ChengZhou17})
\begin{align*}
    \int_{M}u(\Delta_{\phi}v) e^{-\phi}\,dV = \int_{M}v(\Delta_{\phi}u) e^{-\phi}\,dV  = -\int_{M}g(\nabla u,\nabla v)e^{-\phi}\,dV.
\end{align*}

Given a potential function $V \in C^{\infty}(M)$, consider the Witten-Laplacian with potential $\Delta_{\phi}+V$, we define the corresponding Dirichlet-to-Neumann map $\Lambda: H^{\frac{1}{2}}(\partial M) \to H^{-\frac{1}{2}}(\partial M)$ as follows
\begin{equation}\label{1.5}
    \Lambda(f):=\frac{\partial u}{\partial \nu} \Big|_{\partial M},
\end{equation}
where $\nu$ is the outward unit normal vector to the boundary $\partial M$ and $u$ solves the corresponding Dirichlet problem
\begin{equation}\label{1.4}
    \begin{cases}
        (\Delta_{\phi}+V)u = 0 & \text{in}\ M, \\
        u = f  \quad &\text{on}\ \partial M.
    \end{cases}
\end{equation}
We always assume that $0$ is not a Dirichlet eigenvalue for the operator $\Delta_{\phi}+V$. The Dirichlet-to-Neumann map $\Lambda$ is a self-adjoint pseudodifferential operator of order one (see \cite{LeeUhlm89,Taylor11.2}).

In the present paper, we consider the Steklov eigenvalue problem for the operator $\Delta_{\phi}+V$ as follows
\begin{align}\label{1.4.1}
    \begin{cases}
        (\Delta_{\phi}+V)u=0 \quad & \text{in}\ M, \\
        \Lambda(u|_{\partial M}) = \lambda u & \text{on}\ \partial M.
    \end{cases}
\end{align}
It follows from the variational principle that the spectrum of problem \eqref{1.4.1} consists of a discrete sequence
\begin{align*}
    0 < \lambda_{1} \leqslant \lambda_{2} \leqslant \cdots \leqslant \lambda_{k} \leqslant \cdots \to +\infty
\end{align*}
with each eigenvalue repeated according to its multiplicity. The corresponding eigenfunctions $\{u_k\}_{k \geqslant 1}$ form an orthogonal basis in $L^{2}(\partial M)$.

This problem originates from inverse spectral problems. One hopes to recover the geometry of a manifold from the set of the known data (the spectrum of a differential or pseudodifferential operator). Note that the Steklov spectrum of \eqref{1.4.1} coincides with the spectrum of the Dirichlet-to-Neumann map $\Lambda$, they are physical quantities that can be measured experimentally. Apparently, the knowledge provided by the Dirichlet-to-Neumann map $\Lambda$, i.e., the set of the Cauchy data $\{(u|_{\partial M},\frac{\partial u}{\partial \nu}|_{\partial M})\}$ is equivalent to the information given by all the Steklov eigenvalues of \eqref{1.4.1} and the corresponding eigenfunctions.

Here we briefly recall the classical Steklov problem associated with the Laplace--Beltrami operator as follows:
\begin{equation*}
    \begin{cases}
        \Delta u = 0 \quad & \text{in}\  M, \\
        \frac{\partial u}{\partial \nu} = \tau u & \text{on}\ \partial  M.
    \end{cases}
\end{equation*}
The classical Dirichlet-to-Neumann map $\Lambda$ (also known as the voltage-to-current map) is defined by $\Lambda(f) := \frac{\partial u}{\partial \nu}|_{\partial  M}$, where $u$ solves the corresponding Dirichlet problem: $\Delta u = 0$ in $M$ and $u = f$ on $\partial M$.  This problem is an eigenvalue problem with the spectral parameter in the boundary condition. The study of the spectrum of the classical Dirichlet-to-Neumann map $\Lambda$ was initiated by Steklov (Stekloff) \cite{Stekloff02} in 1902. Note that the Steklov spectrum coincides with that of the classical Dirichlet-to-Neumann map. Eigenvalues and eigenfunctions of the classical Dirichlet-to-Neumann map $\Lambda$ have a number of applications in physics, such as fluid mechanics, heat transmission and vibration problems (see \cite{FoxKuttler83,KopaKrein01}). The Steklov spectrum also plays a fundamental role in mathematical analysis of photonic crystals (see \cite{Kuch01}). The classical Steklov problem has attracted much attention in recent years.

Let $\{\tau_{k}\}_{k\geqslant 1}$ be the classical Steklov eigenvalues, that is, the eigenvalues of the classical Dirichlet-to-Neumann map. It follows from Weyl's law for classical Steklov eigenvalues that the eigenvalue counting function $N(\tau):= \#(k|\tau_{k} \leqslant \tau)$ of the classical Dirichlet-to-Neumann map $\Lambda$ satisfies the following asymptotic formula (see \cite{Sandgren55}):
\begin{equation}\label{1.4.2}
    N(\tau)
    = \frac{\operatorname{Vol}(B^{n-1})}{(2\pi)^{n-1}} \operatorname{Vol}(\partial  M) \tau^{n-1} + o(\tau^{n-1}) \quad \text{as}\ \tau \to +\infty.
\end{equation}
Equivalently,
\begin{equation*}
    \sum_{k=1}^{\infty} e^{-t \tau_{k}}
    = \operatorname{Tr} e^{-t \Lambda} \sim \frac{\Gamma(n) \operatorname{Vol}(B^{n-1})}{(2\pi)^{n-1}t^{n-1}} \operatorname{Vol}(\partial  M) \quad \text{as}\ t \to 0^+, 
\end{equation*}
where $B^{n-1}$ is the $(n-1)$-dimensional Euclidean unit ball and $\Gamma(n)$ is the Gamma function. This shows that one can ``hear'' the volume $\operatorname{Vol}(\partial  M)$ of the boundary from the first term of the above asymptotic expansion. A sharp form of \eqref{1.4.2} was given in \cite{Liu11}.

Generally, for the eigenvalues of the classical Dirichlet-to-Neumann map $\Lambda$ associated with the Laplace--Beltrami operator $\Delta$, the heat trace $\operatorname{Tr} e^{-t \Lambda}$ admits an asymptotic expansion
\begin{equation*}
    \sum_{k=1}^{\infty} e^{-t \tau_{k}}
    = \operatorname{Tr} e^{-t \Lambda}
    = \sum_{k=0}^{n-1} \tilde{a}_{k} t^{-n+k+1} + o(1)\quad \text{as}\ t \to 0^+,
\end{equation*}
where the coefficients $\tilde{a}_k = \int_{\partial  M} \tilde{a}_k(x^{\prime}) \,dS,x^{\prime}\in\partial M,0\leqslant k \leqslant n-1$, are spectral invariants of the classical Dirichlet-to-Neumann map. Liu \cite{Liu15} obtained the first four coefficients $\tilde{a}_k(x^{\prime})\ (k=0,1,2,3)$ (see \cite{PoltSher15} for the first three coefficients by a different method). For other geometric invariants, Liu \cite{Liu22p,Liu19} gave the first four coefficients of the asymptotic expansion of heat trace for the polyharmonic Steklov operator and the first two coefficients of the asymptotic expansion of heat trace for the elastic Dirichlet-to-Neumann map. Liu \cite{Liu22S,Liu21} also gave the first two coefficients of the asymptotic expansions of heat traces of the Stokes operator and the Navier--Lam\'{e} operator. In \cite{WangWang19}, the authors gave the first four relative heat invariants (the differences between the coefficients of asymptotic expansions of heat traces for Schr\"{o}dinger operator and that of Laplace--Beltrami operator). In a joint work with Liu, we obtained the first four coefficients of the asymptotic expansion of heat trace for the magnetic Dirichlet-to-Neumann map \cite{LiuTan23} and gave the first two coefficients of the asymptotic expansion of heat trace for the thermoelastic Dirichlet-to-Neumann map \cite{LiuTan22.2}, we also obtained the first two coefficients of the asymptotic expansions of traces of the thermoelastic operators with the Dirichlet and Neumann conditions \cite{LiuTan22.1} on a Riemannian manifold. We refer the reader to \cite{BranGilk90,Gilkey04,Gilkey95,Gilkey75,GilkeyGrubb98} for more results of this topic.

The classical Dirichlet-to-Neumann map is also connected with the famous Calder\'{o}n problem (see \cite{Cald80}), which is to determine the conductivity in the interior of a conductor by boundary measurements from the knowledge of the map $\Lambda$. After the celebrated paper \cite{SylvUhlm87}, plenty of similar problems for elliptic equations have been intensively investigated (see \cite{Uhlmann14}).

It is natural to raise the following interesting problem: what geometric information about manifolds can be explicitly obtained by providing the Steklov spectrum of \eqref{1.4.1}? Once this problem is solved, one can recover the geometry of a manifold from the Steklov spectrum. In this paper, we give an affirmative answer for this problem. This is analogous to the well-known Kac's problem \cite{Kac66}, i.e., is it possible to ``hear'' the shape of a domain just by ``hearing'' all the eigenvalues of the Dirichlet Laplacian? (See \cite{Loren47,Protter87,Weyl12}).

To obtain more geometric information about the manifold, we study the Steklov problem \eqref{1.4.1} and the asymptotic expansion of $\operatorname{Tr} e^{-t\Lambda}$ (the so-called ``heat kernel method''). The coefficients of the asymptotic expansion are spectral invariants, which are metric invariants of the boundary of the manifold, and they contain a lot of geometric and topological information about the manifold (see \cite{Edward91,Gilkey75,Kac66}). Spectral invariants are of great importance in spectral geometry, and they are also connected with many physical concepts (see \cite{ANPS09,Full94}) and have increasingly extensive applications in physics since they describe corresponding physical phenomena. The more the spectral invariants we get, the more the geometric and topological information about the manifold we know. However, computations of spectral invariants are challenging problems (see \cite{Gilkey75,Grubb86,Liu11,SafarovVass97}). Let us point out that the complexity of the operator $\Delta_{\phi}+V$ directly leads to that the calculations of some quantities (for example, spectral invariants) become much more difficult than those of the Laplace--Beltrami operator. To the best of our knowledge, it has not previously been systematically applied in the context of the Steklov problem \eqref{1.4.1}. By the theory of pseudodifferential operators and symbol calculus, we can calculate all the coefficients $a_k\ (0 \leqslant k \leqslant n-1)$ of the asymptotic expansion. In particular, we explicitly give the first four coefficients $a_0$, $a_1$, $a_2$, and $a_3$.

\subsection{Notations and main result}

We denote by $\tilde{R}_{jklm}$, $\tilde{R}_{jk}$, and $\tilde{R}$ the Riemann curvature tensor, the Ricci tensor, and the scalar curvature of $M$, respectively, and denote by  $R_{\alpha\beta\gamma\rho}$, $R_{\alpha\beta}$, $R$, $h_{\alpha\beta}$, $\kappa_\alpha$, and $H := \sum_\alpha \kappa_\alpha$, the Riemann curvature tensor, the Ricci tensor, the scalar curvature, the second fundamental form, the principal curvatures, and the mean curvature of the boundary $\partial M$, respectively. The covariant derivatives $\nabla_{\frac{\partial}{\partial x_l}}\tilde{R}_{jk}:=(\nabla_{\frac{\partial}{\partial x_l}}\operatorname{\widetilde{Ric}})\big(\frac{\partial}{\partial x_j},\frac{\partial}{\partial x_k}\big)$. The partial derivatives $\phi_{k}:=\frac{\partial \phi}{\partial x_k}$, $\phi_{jk}:=\frac{\partial^2 \phi}{\partial x_j \partial x_k}$, and $\phi_{jkl}:=\frac{\partial^3 \phi}{\partial x_j \partial x_k \partial x_l}$. Let the Roman indices run from 1 to $n$, whereas the Greek indices run from 1 to $n-1$, unless otherwise specified. We also denote by $\Gamma(n)$ the Gamma function and $\omega_{n-2}=\operatorname{Vol}(\mathbb{S}^{n-2})=2\pi^{(n-1)/2}/\Gamma((n-1)/2)$ the volume of the $(n-2)$-dimensional Euclidean unit sphere $\mathbb{S}^{n-2}$.

The main result of this paper is the following theorem.
\begin{theorem}\label{thm1.1}
    Suppose that $(M,g)$ is a compact connected Riemannian manifold of dimension $n$ with smooth boundary $\partial M$. Let $\{\lambda_k\}_{k \geqslant 1}$ be the eigenvalues of the Dirichlet-to-Neumann map $\Lambda$ associated to the Witten-Laplacian with potential $\Delta_{\phi}+V$. Then the heat trace of $\Lambda$ admits an asymptotic expansion
    \begin{equation*}
        \sum_{k=1}^{\infty} e^{-t \lambda_{k}}
        = \operatorname{Tr} e^{-t \Lambda}
        = \sum_{k=0}^{n-1} a_{k} t^{-n+k+1} + o(1)\quad \text{as}\ t \to 0^+, 
    \end{equation*}
    where the coefficients $a_k = \int_{\partial M} a_k(x^{\prime})\,dS,\,x^{\prime}\in\partial M,\,0 \leqslant k \leqslant n-1$, are spectral invariants of $\Lambda$, which can be explicitly calculated.
    
    In particular, the first four coefficients are given by
    \begin{align}
        a_0(x^{\prime}) & = \frac{\omega_{n-2}\Gamma(n-1)}{(2\pi)^{n-1}}, \quad n \geqslant 1, \label{a0}\\
        a_1(x^{\prime}) & = \frac{\omega_{n-2}\Gamma(n-1)}{2(2\pi)^{n-1}}\Big(\frac{n-2}{n-1}H+\phi_{n}\Big), \quad n \geqslant 2, \label{a1}\\
        a_2(x^{\prime}) & = \frac{\omega_{n-2}\Gamma(n-2)}{4(2\pi)^{n-1}}
        \biggl[
            \frac{n^3-4n^2+n+8}{2(n^2-1)}H^2 + \frac{n(n-3)}{2(n^2-1)} \sum_{\alpha=1}^{n-1} \kappa_{\alpha}^2  \notag\\
            & \quad + \frac{n-2}{2(n-1)}\tilde{R} - \frac{n-4}{6(n-1)}R + \frac{n-2}{2}\phi_{n}^2 + \frac{n^2-5n+5}{n-1}H\phi_{n}  \notag\\
            & \quad + \Delta\phi - \frac{1}{2}|\nabla\phi|^2 + 2V 
        \biggr], \quad n \geqslant 3, \label{a2}\\
        a_3(x^{\prime}) & = \frac{\omega_{n-2}\Gamma(n-3)}{8(2\pi)^{n-1}}
        \biggl[
            \frac{n^5-5n^4-10n^3+52n^2+2n-114}{6(n^2-1)(n+3)}H^3 \notag\\
            & \quad + \frac{n^4-n^3-12n^2+22n+6}{2(n^2-1)(n+3)}H\sum_{\alpha=1}^{n-1} \kappa_{\alpha}^2 - \frac{4(n-3)(n-2)}{3(n^2-1)(n+3)}\sum_{\alpha=1}^{n-1} \kappa_{\alpha}^3  \notag\\
            & \quad + \frac{n^3-6n^2+2n+14}{2(n^2-1)}H \tilde{R} - \frac{3n^3-20n^2+12n+42}{6(n^2-1)}HR \notag\\
            & \quad + \frac{2(n^2-3n+1)}{n^2-1}\sum_{\alpha=1}^{n-1} \kappa_{\alpha}\tilde{R}_{\alpha\alpha} - \frac{2n(3n-8)}{3(n^2-1)}\sum_{\alpha=1}^{n-1} \kappa_{\alpha}R_{\alpha\alpha} + \frac{n-2}{n-1} \nabla_{\frac{\partial}{\partial x_n}}\tilde{R}_{nn} \notag\\
            & \quad + \frac{n^4-9n^3+20n^2+7n-31}{2(n^2-1)}H^2\phi_n  + \frac{(n-3)(n^2-6n+6)}{2(n-1)}H\phi_n^2 \notag\\
            & \quad + \frac{(n-2)(n-3)}{6}\phi_n^3 + \frac{n^3-7n^2+9n+1}{2(n^2-1)}\phi_n\sum_{\alpha=1}^{n-1} \kappa_{\alpha}^2 + \sum_{\alpha=1}^{n-1} \phi_{\alpha}^2\kappa_{\alpha} - H\phi_{nn} \notag\\
            & \quad + \frac{n^2-6n+7}{2(n-1)}\tilde{R}\phi_n  - \frac{n^2-10n+15}{6(n-1)}R\phi_n + \frac{(n-2)(n-3)}{6(n-1)}\sum_{\alpha=1}^{n-1} \phi_{n\alpha\alpha} \notag\\
            & \quad + \Big(\frac{\partial }{\partial x_n} + (n-3)\phi_n + (n-4)H\Big) \Big(\Delta\phi - \frac{1}{2}|\nabla\phi|^2 + 2V\Big) 
        \biggr], \quad n \geqslant 4. \label{a3}
    \end{align}
\end{theorem}
    
\begin{remark}
    Theorem {\rm\ref{thm1.1}} generalizes the corresponding results for Laplace--Beltrami operator $($see {\rm \cite{Liu15,PoltSher15}}$)$ and for the perturbed polyharmonic operator $(-\Delta)^m+V$ for $m=1$ $($or Schr\"{o}dinger operator$)$ on a Riemannian manifold with boundary $($see {\rm \cite{Liu22p,WangWang19}}$)$.
\end{remark}

The expressions for $a_2(x^{\prime})$ and $a_3(x^{\prime})$ in Theorem \ref{thm1.1} can be simplified when the manifold has constant sectional curvature. Thus, we have the following corollary.
\begin{corollary}\label{cor1.2}
    Assume that the manifold $M$ has constant sectional curvature $K_0$, we then obtain
    \begin{align}
        a_2(x^{\prime}) & = \frac{\omega_{n-2}\Gamma(n-2)}{4(2\pi)^{n-1}}
        \biggl[
            \frac{(n-2)(n^2-n-4)}{2(n^2-1)}H^2 - \frac{2(n^2-3n-1)}{3(n^2-1)}R \notag\\
            & \quad + \frac{n(n-1)(n-2)}{n+1}K_0 + \frac{n-2}{2}\phi_{n}^2 + \frac{n^2-5n+5}{n-1}H\phi_{n}  \notag\\
            & \quad + \Delta\phi - \frac{1}{2}|\nabla\phi|^2 + 2V 
        \biggr], \quad n \geqslant 3, \label{b2}\\
        a_3(x^{\prime}) &= \frac{\omega_{n-2}\Gamma(n-3)}{8(2\pi)^{n-1}}
        \biggl[
            \frac{n^5-2n^4-25n^3+12n^2+164n-96}{6(n^2-1)(n+3)} H^3 \notag\\
            & \quad + \frac{n(n-2)(3n^4-12n^3-38n^2+108n-21)}{3(n^2-1)(n+3)} H K_0 \notag\\
            & \quad - \frac{3n^4-13n^3-44n^2+120n+72}{3(n^2-1)(n+3)} H R + \frac{2(3n^3-n^2-14n-12)}{3(n^2-1)(n+3)} \sum_{\alpha=1}^{n-1} \kappa_\alpha^3 \notag\\
            & \quad + \frac{(n-3)(n-5)(n^2-2)}{2(n^2-1)}H^2\phi_n  + \frac{(n-3)(n^2-6n+6)}{2(n-1)}H\phi_n^2 \notag\\
            & \quad + \frac{(n-2)(n-3)}{6}\phi_n^3 + \sum_{\alpha=1}^{n-1} \phi_{\alpha}^2\kappa_{\alpha} - H\phi_{nn} - \frac{2n^3-15n^2+16n+9}{3(n^2-1)}R\phi_n \notag\\
            & \quad + \frac{(n-1)(n^3-6n^2+6n+1)}{(n+1)}K_0\phi_n + \frac{(n-2)(n-3)}{6(n-1)}\sum_{\alpha=1}^{n-1} \phi_{n\alpha\alpha} \notag\\
            & \quad + \Big(\frac{\partial }{\partial x_n} + (n-3)\phi_n + (n-4)H\Big) \Big(\Delta\phi - \frac{1}{2}|\nabla\phi|^2 + 2V\Big) 
        \biggr], \quad n \geqslant 4. \label{b3}
    \end{align}
\end{corollary}

The main ideas of this paper are as follows. Since the Dirichlet-to-Neumann map $\Lambda$ is a pseudodifferential operator defined on the boundary $\partial M$, we study this kind of operator by the theory of pseudodifferential operators and symbol calculus. There is an effective method (see \cite{Cek20,Grubb86,Hormander85.3}) to calculate its full symbol explicitly. By this method, we flatten the boundary and induce a Riemannian metric in a neighborhood of the boundary. Thus, we obtain a local representation for the Dirichlet-to-Neumann map
\begin{align*}
    \Lambda(u|_{\partial M}) = -\frac{\partial u}{\partial x_n}\Big|_{\partial M}
\end{align*}
in boundary normal coordinates. 

We then look for the following factorization: there exists a pseudodifferential operator $W=W(x,\partial_{x^\prime})$ of order one in $x^\prime\in\partial M$ depending smoothly on $x_n$ such that
\begin{align*}
    \Delta_{\phi}+V
    = \frac{\partial^2 }{\partial x_n^2} + B \frac{\partial }{\partial x_n} + C
    = \Big(\frac{\partial }{\partial x_n} + B - W\Big)\Big(\frac{\partial }{\partial x_n} + W\Big),
\end{align*}
where $B=B(x)$ and $C=C(x,\partial_{x^\prime})$ is a differential operator. As a result, we obtain the equation
\begin{align*}
    W^2 - BW - \Bigl(\frac{\partial }{\partial x_n}  W - W \frac{\partial }{\partial x_n}\Bigr) + C = 0.
\end{align*}
Let $b=b(x)$, $c=c(x,\xi^\prime)$, and $w=w(x,\xi^\prime)$ be the full symbols of $B$, $C$, and $W$, respectively. We then solve the corresponding full symbol equation
\begin{align*}
    \sum_{J} \frac{(-i)^{|J|}}{J !} \partial_{\xi^{\prime}}^{J}w \, \partial_{x^\prime}^{J}w - bw - \frac{\partial w}{\partial x_n} + c = 0,
\end{align*}
where the sum is over all multi-indices $J$, $x^{\prime}=(x_1,\dots,x_{n-1})$ and $\xi^{\prime}=(\xi_1,\dots,\xi_{n-1})$. Note that $W(x,\partial_{x^\prime})$ is a pseudodifferential operator that is independent of $\partial_{x_n}$, but it depends on $x_n$ by construction. Hence, we get the full symbol of $W$, which implies that $W$ is obtained on the boundary $\partial M$ modulo a smoothing operator. 

The trace $\operatorname{Tr} e^{-t \Lambda}$, $t>0$, of the associated heat kernel $\mathcal{K}(t,x^{\prime},y^{\prime})$ admits an asymptotic expansion (see \cite{Agra87,DuisGuill75,GrubbSeeley95,Liu15}), in local coordinates $(x^{\prime},\xi^{\prime})$,
\begin{align*}
    \sum_{k=1}^{\infty} e^{-t \lambda_{k}}
    & = \int_{\partial M} \mathcal{K}(t,x^{\prime},x^{\prime}) \,dS \\
    & = \int_{\partial M}
    \bigg[ 
        \frac{1}{(2\pi)^{n-1}} \int_{T^{*}_{x^\prime}(\partial M)} e^{i \langle x^{\prime} - x^{\prime},\xi^{\prime} \rangle}
        \bigg( 
            \frac{i}{2\pi}\int_{\mathcal{C}} e^{-t\tau}\sigma((\Lambda-\tau)^{-1}) \,d\tau
        \bigg) \,d\xi^{\prime} 
    \bigg] \,dS \\
    & = \int_{\partial M} 
    \bigg[ 
        \frac{1}{(2\pi)^{n-1}} \int_{T^{*}_{x^\prime}(\partial M)}
        \bigg(
            \frac{i}{2\pi}\int_{\mathcal{C}} e^{-t\tau}\sum_{j \leqslant -1} s_{j}(x^{\prime},\xi^{\prime},\tau) \,d\tau 
        \bigg) \,d\xi^{\prime}
    \bigg] \,dS \\
    & \sim \sum_{k=0}^{\infty} a_k t^{-n+k+1} + \sum_{l=0}^{\infty} b_{l} t^l \log t \quad \text{as}\ t \to 0^{+},
\end{align*}
where $T^{*}_{x^\prime}(\partial M)$ is the cotangent space at the point $x^\prime\in\partial M$, $\mathcal{C}$ is a contour around the positive real axis, and $\sigma((\Lambda-\tau)^{-1})\sim\sum_{j \leqslant -1} s_{j}(x^{\prime},\xi^{\prime},\tau)$ is the full symbol of the pseudodifferential operator $(\Lambda-\tau)^{-1}$.

By lots of complicated calculations, we finally establish an effective procedure to calculate all the coefficients of the asymptotic expansion of heat trace and explicitly give the first four coefficients. Note that the method of calculating the coefficients in this paper is different from \cite{Liu15,PoltSher15}. Indeed, inspired by \cite{Liu22p}, at the origin $x_0\in\partial M$ in boundary normal coordinates, we separate the coefficients into two parts, i.e.,
\begin{align*}
    a_{k}(x_0) = \tilde{a}_{k}(x_0) + a_{k}(\phi,V)
\end{align*}
for $0 \leqslant k \leqslant n-1$, where $\tilde{a}_{k}(x_0)$ are the spectral invariants of the classical Dirichlet-to-Neumann map associated with the Laplace--Beltrami operator and $a_{k}(\phi,V)$ only involve $\phi$ and $V$ at $x_0$. Finally, we obtain $a_{k}(x^\prime)$ for any $x^\prime\in\partial M$ since $x_0$ is arbitrary. This method in the present paper is not standard but more effective than that of direct computation.

This paper is organized as follows. In Section \ref{s2}, we briefly introduce some basic concepts of differential operators, curvatures, boundary normal coordinates, and pseudodifferential operators. Moreover, we give the explicit expression for the Dirichlet-to-Neumann map in boundary normal coordinates. In Section \ref{s3}, we derive the factorization of the operator $\Delta_{\phi}+V$, get the full symbols of some pseudodifferential operators, and give the procedure to compute the coefficients of the asymptotic expansion of heat trace. In Section \ref{s4}, we compute the first two coefficients $a_0(x^\prime)$ and $a_1(x^\prime)$. In Sections \ref{s5} and \ref{s6}, we compute $a_2(x^\prime)$ and $a_3(x^\prime)$, respectively. Finally, Section \ref{s7} is devoted to proving Corollary \ref{cor1.2}.

\section{Preliminaries}\label{s2}

In this section, we will briefly introduce some differential operators, curvatures (cf. \cite[Chap.\,1]{ChowLL06}), boundary normal coordinates (see \cite{LeeUhlm89} or \cite[p.\,532]{Taylor11.2}), pseudodifferential operators and symbols (cf. \cite[Chap.\,7]{Taylor11.2}).

\subsection{Differential operators and curvatures}

First, we introduce some differential operators and curvatures in Riemannian geometry (cf. \cite[Chap.\,1]{ChowLL06}).

Let $(M,g)$ be a smooth Riemannian manifold of dimension $n$. In the local coordinates $\{x_j\}_{j=1}^n$, we denote by $\big\{\frac{\partial}{\partial x_j}\big\}_{j=1}^n$ and $\{dx_j\}_{j=1}^n$, respectively, the natural basis for the tangent space $T_x M$ and the cotangent space $T_x^{*} M$ at the point $x\in M$. Then, the Riemannian metric $g$ is given by  $g = \sum_{j,k} g_{jk} \,dx_j\otimes dx_k$.

For two smooth vector fields $X= \sum_{j} X^j \frac{\partial}{\partial x_j}$ and $Y=\sum_{j}  Y^j \frac{\partial}{\partial x_j}$, the inner product with respect to the metric $g$ is denoted by
\begin{equation}\label{1.1.1}
    g(X,Y) =\sum_{j,k} g_{jk} X^j Y^k.
\end{equation}
Let $\nabla_{\frac{\partial}{\partial x_j}}$ be the covariant derivative with respect to $\frac{\partial}{\partial x_j}$. Then,
\begin{align}\label{9.19}
    \nabla_{\frac{\partial}{\partial x_j}}\frac{\partial}{\partial x_k}=\sum_{l}  \Gamma^l_{jk}\frac{\partial}{\partial x_l}.
\end{align}
Here the Christoffel symbols
\begin{align}\label{chris}
    \Gamma^{l}_{jk} 
    = \frac{1}{2} \sum_{m} g^{lm}(g_{jm,k} + g_{km,j} - g_{jk,m}),
\end{align}
where $g_{jk,m}=\frac{\partial g_{jk}}{\partial x_m}$. From \eqref{chris} we have
\begin{align}
    \label{chr1} \sum_{m} \Gamma^{m}_{jk} g_{lm}
    &= \frac{1}{2} (g_{jl,k} + g_{kl,j} - g_{jk,l}).
\end{align}

The gradient operator is denoted by
\begin{align*}
    \nabla u
    = \sum_{j,k} g^{jk} \frac{\partial u}{\partial x_j}\frac{\partial}{\partial x_k} , \quad u \in C^{\infty}( M),
\end{align*}
where $(g^{jk}) = (g_{jk})^{-1}$. The Laplace--Beltrami operator is given by
\begin{align*}
    \Delta u
    =  \sum_{j,k}g^{jk}\bigg(\frac{\partial^2 u}{\partial x_j \partial x_k} - \sum_{l}\Gamma^l_{jk}\frac{\partial u}{\partial x_l}\bigg), \quad u \in C^{\infty}( M).
\end{align*}

The curvature tensor $\mathcal{R}$ of the Riemannian manifold is given by
\begin{align*}
    \mathcal{R}\Big(\frac{\partial}{\partial x_j},\frac{\partial}{\partial x_k}\Big)\frac{\partial}{\partial x_l}:=\sum_{m}R^m_{jkl}\frac{\partial}{\partial x_m},
\end{align*}
where
\begin{align*}
    R^{m}_{jkl} 
    = \frac{\partial \Gamma^{m}_{kl}}{\partial x_j} - \frac{\partial \Gamma^{m}_{jl}}{\partial x_k} + \sum_{p}(\Gamma^{p}_{kl} \Gamma^{m}_{jp} - \Gamma^{p}_{jl} \Gamma^{m}_{kp}).
\end{align*}

The Riemann curvature tensor is defined by
\begin{align*}
    R_{jklm}
    =R\Bigl(\frac{\partial}{\partial x_j},\frac{\partial}{\partial x_k},\frac{\partial}{\partial x_l},\frac{\partial}{\partial x_m}\Bigr)
    :=g\Big( \mathcal{R}\Big(\frac{\partial}{\partial x_j},\frac{\partial}{\partial x_k}\Big)\frac{\partial}{\partial x_l}, \frac{\partial}{\partial x_m}\Big).
\end{align*}
Some basic symmetries of the Riemann curvature tensor are
\begin{align*}
    R_{jklm}=R_{lmjk}=-R_{kjlm}=-R_{jkml}.
\end{align*}
In local coordinates,
\begin{align}\label{rm}
    R_{jklm}
    &=\sum_{p}g_{mp} R^{p}_{jkl} \notag\\
    &=\frac{1}{2}(g_{jl,km}+g_{km,jl}-g_{jm,kl}-g_{kl,jm})+ \sum_{p,h}g_{ph}(\Gamma_{jl}^p\Gamma_{km}^h-\Gamma_{jm}^p\Gamma_{kl}^h).
\end{align}

If $P\subset T_x M$ is a 2-plane, then the sectional curvature of $P$ is defined by
\begin{align*}
    K(P):= R(e_1,e_2,e_2,e_1),
\end{align*}
where $\{e_1,e_2\}$ is an orthonormal basis of $P$. This definition is independent of the choice of such a basis. A Riemannian manifold $M$ has constant sectional curvature if the sectional curvature of every 2-plane is the same. That is, there exists $K_0\in\mathbb{R}$ such that for every $x\in M$ and 2-plane $P\subset T_x M$, $K(P)=K_0$.

The the Ricci tensor is given by
\begin{align}\label{ricci}
    \operatorname{Ric}\Big(\frac{\partial}{\partial x_j},\frac{\partial}{\partial x_k}\Big)=R_{jk} = \sum_{l,m} g^{lm} R_{jlmk}.
\end{align}
The covariant derivative $\nabla_{\frac{\partial}{\partial x_l}} R_{jk}$ of the Ricci tensor is defined by
\begin{align}\label{covariant}
    \nabla_{\frac{\partial}{\partial x_l}} R_{jk}
    &:=(\nabla\operatorname{Ric})\Big(\frac{\partial}{\partial x_l},\frac{\partial}{\partial x_j},\frac{\partial}{\partial x_k}\Big)
    =(\nabla_{\frac{\partial}{\partial x_l}}\operatorname{Ric})\Big(\frac{\partial}{\partial x_j},\frac{\partial}{\partial x_k}\Big) \notag\\
    &= \frac{\partial R_{jk}}{\partial x_l} - \sum_{m}\Gamma^{m}_{lj}R_{mk} - \sum_{m}\Gamma^{m}_{kl}R_{jm}.
\end{align}
The scalar curvature is the trace of the Ricci tensor, that is,
\begin{align}\label{scalar}
    R = \sum_{j,k}g^{jk} R_{jk}.
\end{align}

\subsection{Boundary normal coordinates}

Next, we briefly introduce the construction of geodesic coordinates with respect to the boundary (see \cite{LeeUhlm89} or \cite[p.\,532]{Taylor11.2}). 

Let $( M,g)$ be a smooth compact Riemannian manifold of dimension $n$ with smooth boundary $\partial  M$. For each boundary point $x^{\prime} \in \partial  M$, let $\gamma_{x^{\prime}}:[0,\varepsilon)\to \bar{ M}$ denote the unit-speed geodesic starting at $x^{\prime}$ and normal to the boundary $\partial  M$. If $ x^{\prime} := \{x_{1}, \ldots, x_{n-1}\}$ are any local coordinates for the boundary $\partial  M$ near $x_0 \in \partial  M$, we can extend them smoothly to functions on a neighborhood of $x_0$ in $\bar{ M}$ by letting them be constant along each normal geodesic $\gamma_{x^{\prime}}$. If we then define $x_n$ to be the parameter along each $\gamma_{x^{\prime}}$, it follows easily that $\{x_{1}, \ldots, x_{n}\}$ form coordinates for $\bar{ M}$ in some neighborhood of $x_0$, which we call the boundary normal coordinates determined by $\{x_{1}, \ldots, x_{n-1}\}$. In these coordinates, $x_n>0$ in $ M$, and the boundary $\partial  M$ is locally characterized by $x_n=0$. A standard computation shows that the metric $g$ then has the form 
\begin{align*}
    g =\sum_{\alpha,\beta} g_{\alpha\beta} \,dx_{\alpha} \,dx_{\beta} + dx_{n}^{2}.
\end{align*}
This implies that $g_{n\alpha}=g_{\alpha n}=0$, $1 \leqslant \alpha \leqslant n-1$, and $g_{nn}=1$ in a neighborhood of $x_0$ in $ M$. Therefore, in boundary normal coordinates, we have $g_{jk}(x_{0}) = g^{jk}(x_{0}) = \delta_{jk}$ and $g_{jk,\alpha}(x_{0}) = g^{jk,\alpha}(x_{0}) = 0$.

The second fundamental form $h$ is defined by
\begin{align*}
    h(X,Y):= g(\nabla_X \nu,Y),
\end{align*} 
where $\nu$ is the outward unit normal vector to the boundary $\partial  M$. Applying \eqref{9.19}, \eqref{chr1} and $g_{n\alpha}=0$ in a neighborhood of $x_0$ in $ M$, the $(\alpha,\beta)$ entry of the second fundamental form $h$ is given by
\begin{align}\label{6.1}
    h_{\alpha\beta}
    &:= g\Big( \nabla_{\frac{\partial }{\partial x_{\alpha}}} \nu, \frac{\partial }{\partial x_{\beta}} \Big)
    = g\Big( \nabla_{\frac{\partial }{\partial x_{\alpha}}} \Bigl(-\frac{\partial }{\partial x_{n}}\Bigr), \frac{\partial }{\partial x_{\beta}} \Big)\notag\\
    &=-\sum_{j}\Gamma_{n\alpha}^j g_{j\beta}
    = -\frac{1}{2}g_{\alpha\beta,n}.
\end{align}

We can choose a frame that diagonalizes the second fundamental form such that the $(\alpha,\beta)$ entry of $h$ at the point $x_{0}$ is
\begin{align*}
    h_{\alpha\beta}(x_0)=-\frac{1}{2}g_{\alpha\beta,n}(x_0)=\kappa_\alpha\delta_{\alpha\beta},
\end{align*}
where $\kappa_\alpha$, $1\leqslant \alpha \leqslant n-1$, are the principal curvatures at $x_0\in\partial  M$, and $\delta_{jk}$ is the Kronecker delta. The mean curvature of the boundary is defined by
\begin{align}\label{6.2}
    H:=\sum_\alpha \kappa_\alpha.
\end{align}
It is easy to show that $g_{jk,n}(x_{0}) = -g^{jk,n}(x_{0})$ since $\sum_k g_{jk}g^{kl}=\delta_{jl}$. Therefore, in boundary normal coordinates, for $1 \leqslant j,k \leqslant n$ and $1 \leqslant \alpha,\beta \leqslant n-1$, we conclude that
\begin{align}
    \label{6.3} g_{jk}(x_{0}) &= g^{jk}(x_{0}) = \delta_{jk}, \\
    \label{6.3.1} g_{jk,\alpha}(x_{0}) &= g^{jk,\alpha}(x_{0}) = 0,\\
    \label{6.3.2} g_{\alpha\beta,n}(x_{0}) &= -g^{\alpha\beta,n}(x_{0}) = -2\kappa_\alpha \delta_{\alpha\beta}.
\end{align}

\begin{lemma}\label{lem2.1}
    Let $\tilde{R}$ and $R$ be the scalar curvatures of $M$ and $\partial M$, respectively. Then at the origin $x_0$ in boundary normal coordinates,
    \begin{align}
        &\sum_\alpha g_{\alpha\alpha,nn}(x_0)=3\sum_\alpha\kappa_\alpha^2-H^2-\tilde{R}+R,\\
        &\sum_\alpha g^{\alpha\alpha,nn}(x_0)=5\sum_\alpha\kappa_\alpha^2+H^2+\tilde{R}-R.
    \end{align}
\end{lemma}

\begin{proof}
    It follows from \eqref{chris}, \eqref{rm}, and \eqref{6.3.2} that, at the origin $x_0$ in boundary normal coordinates,
    \begin{align*}
        &\tilde{R}_{\alpha n\alpha n}(x_0)=\frac{1}{2}g_{\alpha\alpha,nn}-\sum_\beta (\Gamma_{\alpha n}^\beta)^2,\\
        &\Gamma_{\alpha n}^\beta(x_0)=-\kappa_\alpha\delta_{\alpha\beta}.
    \end{align*}
    Thus,
    \begin{align}\label{c1}
        \sum_\alpha g_{\alpha\alpha,nn}(x_0)=2\sum_\alpha\kappa_\alpha^2-2\tilde{R}_{nn}.
    \end{align}
    According to the Gauss equation (see \cite[p.\,39--40]{ChowLL06})
    \begin{align*}
        R_{\alpha\beta\gamma\rho} 
        = \tilde{R}_{\alpha\beta\gamma\rho} + h_{\alpha\rho}h_{\beta\gamma} - h_{\alpha\gamma}h_{\beta\rho}.
    \end{align*}
    Using \eqref{ricci}, \eqref{scalar}, and \eqref{6.2} in boundary normal coordinates, we get
    \begin{align*}
        R_{\alpha\rho}(x_0)
        &=\sum_\beta R_{\alpha\beta\beta\rho}
        =\sum_\beta (\tilde{R}_{\alpha\beta\beta\rho} + h_{\alpha\rho}h_{\beta\beta} - h_{\alpha\beta}h_{\beta\rho}) \notag\\
        &= \tilde{R}_{\alpha\rho} - \tilde{R}_{n \alpha\rho n} + Hh_{\alpha\rho} - \sum_{\beta} h_{\alpha\beta}h_{\beta\rho}\notag\\
        &= \tilde{R}_{\alpha\rho} - \tilde{R}_{n \alpha\rho n} + H\kappa_\alpha\delta_{\alpha\rho} -  \kappa_\alpha\kappa_\rho \delta_{\alpha\rho}.
    \end{align*}
    Hence,
    \begin{align}
        \label{4.1} &R_{\alpha\alpha}(x_0)
        =\tilde{R}_{\alpha\alpha} - \tilde{R}_{n \alpha\alpha n} + H\kappa_\alpha - \kappa_\alpha^2,\\
        \label{4.2} &R(x_0)
        =\sum_\alpha R_{\alpha\alpha}
        = \tilde{R} - 2\tilde{R}_{nn} + H^2 - \sum_\alpha \kappa_\alpha^2.
    \end{align}
    Since $\sum_k g_{\alpha k}g^{k\beta}=\delta_{\alpha\beta}$, we have
    \begin{align*}
        0=\frac{\partial^2}{\partial x_n^2}\Big(\sum_k g_{\alpha k}g^{k\alpha}\Big)
        =\sum_k\Big(g_{\alpha k,nn}g^{k\alpha}+2g_{\alpha k,n}g^{k\alpha,n}+g_{\alpha k}g^{k\alpha,nn}\Big).
    \end{align*}
    Thus,
    \begin{align*}
        &\sum_\alpha g_{\alpha\alpha,nn}(x_0)+\sum_\alpha g^{\alpha\alpha,nn}(x_0)=8\sum_\alpha\kappa_\alpha^2.
    \end{align*}
    Therefore, the desired results follows from the above equality and \eqref{c1}, \eqref{4.2}.
\end{proof}

\subsection{Pseudodifferential operators and symbols}

Finally, we recall some concepts of pseudodifferential operators and their symbols (cf. \cite[Chap.\,7]{Taylor11.2}). 

Assuming $U \subset \mathbb{R}^n$ and $m \in \mathbb{R}$, we define $S^{m}_{1,0} = S^{m}_{1,0}(U,\mathbb{R}^n)$ to consist of $C^{\infty}$-functions $p(x,\xi)$ satisfying for every compact set $Q \subset U$,
\begin{align*}
    |D_{x}^{\beta}D_{\xi}^{\alpha}p(x,\xi)| 
    \leqslant C_{Q,\alpha,\beta} \langle \xi \rangle ^{m-|\alpha|}, \quad x \in Q,\ \xi \in \mathbb{R}^n
\end{align*}
for all $\alpha,\beta \in \mathbb{N}^{n}$, where $D^{\alpha} = D^{\alpha_1} \cdots D^{\alpha_n}$, $D_{j} = -i \frac{\partial }{\partial x_j}$ and $\langle \xi \rangle = (1 + |\xi|^2)^{1/2}$. The elements of $S^{m}_{1,0}$ are called symbols of order $m$. It is clear that $S^{m}_{1,0}$ is a Fr\'{e}chet space with semi-norms given by
\begin{align*}
    \|p\|_{Q,\alpha,\beta} 
    := \sup_{x \in Q}
    \big| \big( D_{x}^{\beta}D_{\xi}^{\alpha}p(x,\xi) \big) (1 + |\xi|)^{-m+|\alpha|} \big|.
\end{align*}

Let $p(x,\xi) \in S^{m}_{1,0}$ and $\hat{u}(\xi) = \int_{\mathbb{R}^n}e^{-iy \cdot \xi} u(y) \,dy$ be the Fourier transform of $u$. A pseudodifferential operator in an open set $U$ is essentially defined by a Fourier integral operator
\begin{align*}
    P(x,D)u(x) 
    = \frac{1}{(2\pi)^n}\int_{\mathbb{R}^n}p(x,\xi)e^{ix \cdot \xi} \hat{u}(\xi) \,d\xi
\end{align*}
for $u \in C^{\infty}_{0}(U)$. In such a case, we say the associated operator $P(x,D)$ belongs to $OPS^{m}$. We denote $OPS^{-\infty}=\bigcap_m OPS^m$. If there are smooth $p_{m-j}(x,\xi)$, homogeneous in $\xi$ of degree $m-j$ for $|\xi| \geqslant 1$, that is, $p_{m-j}(x,r\xi) = r^{m-j}p_{m-j}(x,\xi)$ for $r>0$, and if
\begin{align}\label{symbol}
    p(x,\xi) \sim \sum_{j \geqslant 0}p_{m-j}(x,\xi)
\end{align}
in the sense that
\begin{align*}
    p(x,\xi) - \sum_{j = 0}^{N} p_{m-j}(x,\xi) \in S_{1,0}^{m-N-1}
\end{align*}
for all $N$, then we say $p(x,\xi) \in S_{cl}^{m}$, or just $p(x,\xi) \in S^{m}$. We denote $S^{-\infty}=\bigcap_m S^m$. We call $p_{m}(x,\xi)$ the principal symbol of $P(x,D)$. We say $P(x,D) \in OPS^{m}$ is elliptic of order $m$ if on each compact $Q \subset U$ there are constants $C_{Q}$ and $r < \infty$ such that
\begin{align*}
    |p(x,\xi)^{-1}| \leqslant C_{Q} \langle \xi \rangle ^{-m}, \quad |\xi| \geqslant r.
\end{align*}

We can now define a pseudodifferential operator on a manifold $ M$. In particular,
\begin{align*}
    P:C_{0}^{\infty}( M) \to C^{\infty}( M)
\end{align*}
belongs to $OPS^{m}_{1,0}( M)$ if the kernel of $P$ is smooth off the diagonal in $ M \times  M$ and if for any coordinate neighborhood $U \subset  M$ with $\Phi:U \to \mathcal{O}$ a diffeomorphism onto an open subset $\mathcal{O} \subset \mathbb{R}^{n}$, the map $\tilde{P}:C_{0}^{\infty}(\mathcal{O}) \to C^{\infty}(\mathcal{O})$ given by
\begin{align*}
    \tilde{P}u := P(u \circ \Phi) \circ \Phi^{-1}
\end{align*}
belongs to $OPS^{m}_{1,0}(\mathcal{O})$. We refer the reader to \cite{Grubb86,Hormander85.3,Taylor11.2,Taylor81} for more details.

\section{Symbol calculus and the asymptotic expansion of heat trace}\label{s3}

In this section, we will derive the factorization of the operator $\Delta_{\phi}+V$ and calculate the full symbols of the Dirichlet-to-Neumann map $\Lambda$ and the pseudodifferential operator $(\Lambda-\tau)^{-1}$ which is a pseudodifferential operator of order $-1$ in $\xi^{\prime}$ with parameter $\tau$ (see \cite[Chap.\,II, Sect.\,9]{Shubin01}, \cite[Chap.\,1, Sect.\,1.6]{Gilkey95} or \cite[p.\,16]{Taylor11.2}).

\subsection{Symbols of pseudodifferential operators}

In boundary normal coordinates, we can write the Laplace--Beltrami operator as
\begin{align*}
    \Delta
    &= \frac{\partial^2 }{\partial x_n^2} + \frac{1}{2} \sum_{\alpha,\beta}  g^{\alpha\beta}g_{\alpha\beta,n} \frac{\partial }{\partial x_n} + \sum_{\alpha,\beta}  g^{\alpha\beta} \frac{\partial^2}{\partial x_{\alpha} \partial x_{\beta}} + \sum_{\alpha,\beta}  \biggl(\frac{1}{2} g^{\alpha\beta} \sum_{\gamma,\rho}  g^{\gamma\rho} g_{\gamma\rho,\alpha} + g^{\alpha\beta,\alpha}\biggr)\frac{\partial}{\partial x_\beta}.
\end{align*}
It follows from this and \eqref{1.1}, \eqref{1.1.1} that
\begin{align}\label{2.2}
    \Delta_{\phi}+V
    = \frac{\partial^2 }{\partial x_n^2} + B \frac{\partial }{\partial x_n} + C,
\end{align}
where $C = C_2 + C_1 + C_0$ and
\begin{align*}
    & B = \frac{1}{2}\sum_{\alpha,\beta} g^{\alpha\beta} g_{\alpha\beta,n} - \phi_n, \\
    & C_2 = \sum_{\alpha,\beta}g^{\alpha\beta} \frac{\partial^2}{\partial x_{\alpha} \partial x_{\beta}}, \\
    & C_1 = \sum_{\alpha,\beta}  \bigg(\frac{1}{2} g^{\alpha\beta} \sum_{\gamma,\rho}  g^{\gamma\rho} g_{\gamma\rho,\alpha} + g^{\alpha\beta,\alpha}\bigg)\frac{\partial}{\partial x_\beta} -\sum_{\alpha} \phi_\alpha \frac{\partial}{\partial x_\alpha}, \\
    & C_0 = V.
\end{align*}

We then derive the microlocal factorization of the operator $\Delta_{\phi}+V$.
\begin{proposition}\label{prop2.1}
    There exists a pseudodifferential operator $W(x,\partial_{x^\prime})$ of order one in $x^\prime$ depending smoothly on $x_n$ such that
    \begin{align*}
        \Delta_{\phi}+V = \Big(\frac{\partial }{\partial x_n} + B - W\Big)\Big(\frac{\partial }{\partial x_n} + W\Big)
    \end{align*}
    modulo a smoothing operator.
\end{proposition}

\begin{proof}
    It follows from \eqref{2.2} that
    \begin{align*}
        \frac{\partial^2 }{\partial x_n^2} + B \frac{\partial }{\partial x_n} + C 
        = \Big(\frac{\partial }{\partial x_n} + B - W \Big) \Big(\frac{\partial }{\partial x_n} + W \Big),
    \end{align*}
    that is,
    \begin{align}\label{2.4}
        W^2 - BW - \Bigl(\frac{\partial }{\partial x_n}  W - W \frac{\partial }{\partial x_n}\Bigr) + C = 0.
    \end{align}
    Note that for any smooth function $v$,
    \begin{align*}
        \Bigl(\frac{\partial }{\partial x_n}  W - W \frac{\partial }{\partial x_n}\Bigr)v
        &=\frac{\partial }{\partial x_n}(Wv)-W\frac{\partial v}{\partial x_n}
        =\frac{\partial W}{\partial x_n} v + W\frac{\partial v}{\partial x_n} - W\frac{\partial v}{\partial x_n}\\
        &= \frac{\partial W}{\partial x_n} v.
    \end{align*}

    Recall that if $F$ and $G$ are two pseudodifferential operators with full symbols $f(x,\xi)$ and $g(x,\xi)$ in local coordinates $(x,\xi)$, respectively, then the full symbol $\sigma(FG)(x,\xi)$ of $FG$ is given by (see \cite[p.\,11]{Taylor11.2} or \cite[p.\,71]{Hormander85.3} and also \cite{Grubb86,Treves80})
    \begin{align*}
        \sigma(FG)(x,\xi) \sim \sum_{J} \frac{(-i)^{|J|}}{J !} \partial_{\xi}^{J} f(x,\xi) \, \partial_{x}^{J} g(x,\xi),
    \end{align*}
    where the sum is over all multi-indices $J$. 

    Let $w = w(x,\xi^{\prime})$ be the full symbol of $W$, we write
    \begin{align*}
        w(x,\xi^{\prime}) \sim \sum_{j\leqslant 1} w_j(x,\xi^{\prime})
    \end{align*}
    with $w_j=w_j(x,\xi^{\prime})$ homogeneous of degree $j$ in $\xi^{\prime}$. Let $b(x)$ and $c(x,\xi^{\prime}) = c_2(x,\xi^{\prime}) + c_1(x,\xi^{\prime}) + c_0(x)$ be the full symbols of $B$ and $C$, respectively. We write
    \begin{align}\label{10.1}
        \xi^{\alpha}=\sum_{\beta}g^{\alpha\beta}\xi_{\beta},\quad |\xi^{\prime}|^2=\sum_{\alpha}\xi^{\alpha}\xi_{\alpha}=\sum_{\alpha,\beta}g^{\alpha\beta}\xi_\alpha\xi_\beta.
    \end{align}
    Then,
    \begin{align}
        & b(x) = \frac{1}{2}\sum_{\alpha,\beta} g^{\alpha\beta} g_{\alpha\beta,n} -\phi_n,\label{6.6}\\
        & c_2(x,\xi^{\prime}) = -|\xi^{\prime}|^2,\label{6.7}\\
        & c_1(x,\xi^{\prime}) = i\sum_{\alpha}
        \bigg(
            \frac{1}{2}\xi^{\alpha} \sum_{\gamma,\rho} g^{\gamma\rho} g_{\gamma\rho,\alpha} + \frac{\partial \xi^{\alpha}}{\partial x_\alpha}
        \bigg)
        - i \sum_{\alpha}\phi_\alpha \xi_\alpha, \label{6.8}\\
        & c_0(x) = V.\label{6.9}
    \end{align}

    Noting that $\partial_{\xi^{\prime}}^{J} b(x) = 0$ for all $|J|>0$, we conclude that the full symbol equation of \eqref{2.4} is
    \begin{equation}\label{2.5}
        \sum_{J} \frac{(-i)^{|J|}}{J !} \partial_{\xi^{\prime}}^{J} w \, \partial_{x^\prime}^{J} w - bw - \frac{\partial w}{\partial x_n}  + c = 0,
    \end{equation}
    where $b=b(x)$ and $c=c(x,\xi^\prime)$. Note that the partial derivatives with respect to $x_n$ and $\xi_n$ do not appear in $\partial_{\xi^{\prime}}^{J} w \, \partial_{x^\prime}^{J} w$ in the above equation since $w(x,\xi^{\prime})$ does not depend on $\xi_n$ by construction.

    We shall determine $w_j=w_j(x,\xi^{\prime})$, $j\leqslant 1$, recursively so that \eqref{2.5} holds modulo $S^{-\infty}$. Grouping the homogeneous terms of degree two in \eqref{2.5}, which implies
    \begin{align*}
        w_1^2 + c_2 = 0,
    \end{align*}
    so we can choose
    \begin{align}\label{2.6}
        w_1 = |\xi^{\prime}|.
    \end{align}
    The terms of degree one in \eqref{2.5} are
    \begin{align*}
        2w_1w_0 - i\sum_\alpha \frac{\partial w_1}{\partial \xi_\alpha}\frac{\partial w_1}{\partial x_\alpha} - bw_1 - \frac{\partial w_1}{\partial x_n} + c_1 = 0
    \end{align*}
    so that
    \begin{align}\label{2.7}
        w_0 
        = \frac{1}{2} w_1^{-1} 
        \biggl(
            i\sum_\alpha\frac{\partial w_1}{\partial \xi_\alpha}\frac{\partial w_1}{\partial x_\alpha} + bw_1 + \frac{\partial w_1}{\partial x_n} - c_1
        \biggr).
    \end{align}
    The terms of degree zero in \eqref{2.5} are
    \begin{align*}
        &2w_1w_{-1} + w_0^2 - i\sum_\alpha
        \Big(\frac{\partial w_1}{\partial \xi_\alpha}\frac{\partial w_0}{\partial x_\alpha} + \frac{\partial w_0}{\partial \xi_\alpha}\frac{\partial w_1}{\partial x_\alpha}\Big)
        - \frac{1}{2}\sum_{\alpha,\beta} \frac{\partial^2 w_1}{\partial \xi_{\alpha} \partial \xi_{\beta}}\frac{\partial^2 w_1}{\partial x_{\alpha} \partial x_{\beta}}\\
        &- bw_0 - \frac{\partial w_0}{\partial x_n} + c_0 = 0,
    \end{align*}
    i.e.,
    \begin{align}\label{2.8}
        w_{-1} 
        & = \frac{1}{2} w_1^{-1} 
        \bigg[
            -w_0^2 + i\sum_\alpha \Big(\frac{\partial w_1}{\partial \xi_\alpha}\frac{\partial w_0}{\partial x_\alpha} + \frac{\partial w_0}{\partial \xi_\alpha}\frac{\partial w_1}{\partial x_\alpha}\Big)+ \frac{1}{2}\sum_{\alpha,\beta} \frac{\partial^2 w_1}{\partial \xi_{\alpha} \partial \xi_{\beta}}\frac{\partial^2 w_1}{\partial x_{\alpha} \partial x_{\beta}}  \notag\\
            & \quad + bw_0 + \frac{\partial w_0}{\partial x_n} - c_0 
        \bigg].
    \end{align}
    The terms of degree $-1$ in \eqref{2.5} are
    \begin{align*}
        & 2w_1w_{-2} + 2w_0w_{-1} - i\sum_\alpha \Big(\frac{\partial w_1}{\partial \xi_\alpha}\frac{\partial w_{-1}}{\partial x_\alpha}+\frac{\partial w_0}{\partial \xi_\alpha}\frac{\partial w_0}{\partial x_\alpha} +\frac{\partial w_{-1}}{\partial \xi_\alpha}\frac{\partial w_1}{\partial x_\alpha}\Big) \\
        & - \frac{1}{2} \sum_{\alpha,\beta} \Big(\frac{\partial^2 w_1}{\partial \xi_{\alpha} \partial \xi_{\beta}} \frac{\partial^2 w_0}{\partial x_{\alpha} \partial x_{\beta}} + \frac{\partial^2 w_0}{\partial \xi_{\alpha} \partial \xi_{\beta}} \frac{\partial^2 w_1}{\partial x_{\alpha} \partial x_{\beta}}\Big) + \frac{i}{6}\sum_{\alpha,\beta,\gamma} \frac{\partial^3 w_1}{\partial \xi_{\alpha} \partial \xi_{\beta} \partial \xi_{\gamma}}\frac{\partial^3 w_1}{\partial x_{\alpha} \partial x_{\beta} \partial x_{\gamma}} \\
        & - bw_{-1} - \frac{\partial w_{-1}}{\partial x_n} = 0, 
    \end{align*}
    and we obtain
    \begin{align}\label{2.9}
        w_{-2} 
        & = \frac{1}{2} w_1^{-1} 
        \bigg[
            -2w_0w_{-1} + i\sum_\alpha \Big(\frac{\partial w_1}{\partial \xi_\alpha}\frac{\partial w_{-1}}{\partial x_\alpha}+\frac{\partial w_0}{\partial \xi_\alpha}\frac{\partial w_0}{\partial x_\alpha} +\frac{\partial w_{-1}}{\partial \xi_\alpha}\frac{\partial w_1}{\partial x_\alpha}\Big) \notag\\
            & \quad   + \frac{1}{2} \sum_{\alpha,\beta} \Big(\frac{\partial^2 w_1}{\partial \xi_{\alpha} \partial \xi_{\beta}} \frac{\partial^2 w_0}{\partial x_{\alpha} \partial x_{\beta}} + \frac{\partial^2 w_0}{\partial \xi_{\alpha} \partial \xi_{\beta}} \frac{\partial^2 w_1}{\partial x_{\alpha} \partial x_{\beta}}\Big) \notag\\
            & \quad - \frac{i}{6}\sum_{\alpha,\beta,\gamma} \frac{\partial^3 w_1}{\partial \xi_{\alpha} \partial \xi_{\beta} \partial \xi_{\gamma}}\frac{\partial^3 w_1}{\partial x_{\alpha} \partial x_{\beta} \partial x_{\gamma}} + bw_{-1} + \frac{\partial w_{-1}}{\partial x_n}
        \bigg].
    \end{align}
    Proceeding recursively, for the terms of degree $-m\ (m \geqslant 2)$, we have
    \begin{align*}
        2w_1w_{-1-m} + \sum_{\substack{-m \leqslant j,k \leqslant 1 \\ |J| = j + k + m}} \frac{(-i)^{|J|}}{J!} \partial_{\xi^{\prime}}^{J} w_j \, \partial_{x^\prime}^{J} w_k - bw_{-m} - \frac{\partial w_{-m}}{\partial x_n} = 0,
    \end{align*}
    namely,
    \begin{align}\label{2.9.1}
        w_{-1-m} = -\frac{1}{2} w_1^{-1} 
        \left(
            \sum_{\substack{-m \leqslant j,k \leqslant 1 \\ |J| = j + k + m}} \frac{(-i)^{|J|}}{J !} \partial_{\xi^{\prime}}^{J} w_j \, \partial_{x^\prime}^{J} w_k - bw_{-m} - \frac{\partial w_{-m}}{\partial x_n}
        \right)
    \end{align}
    for $m \geqslant 2$.
\end{proof}

From Proposition \ref{prop2.1}, we get the full symbol $w(x,\xi^{\prime}) \sim \sum_{j\leqslant 1} w_j(x,\xi^{\prime})$ of the pseudodifferential operator $W$. This implies that $W$ is obtained on the boundary $\partial  M$ modulo a smoothing operator.

\begin{proposition}
    In boundary normal coordinates, the Dirichlet-to-Neumann map $\Lambda$ can be represented as
    \begin{align}\label{2.10}
        \Lambda(f) = Wu|_{\partial M}
    \end{align}
    modulo a smoothing operator, where $u$ solves the corresponding Dirichlet problem \eqref{1.4}.
\end{proposition}

\begin{proof}
    In the boundary normal coordinates $(x^\prime,x_n)$ with $x_n\in[0,T]$. Since the principal symbol of the operator $\Delta_{\phi}+V$ is negative, the hyperplane ${x_n = 0}$ is non-characteristic. Hence, $\Delta_{\phi}+V$ is partially hypoelliptic with respect to this boundary (see \cite[p.\,107]{Hormander64}). Therefore, the solution to the equation $(\Delta_{\phi}+V)u = 0$ is smooth in normal variable, that is, $u\in C^\infty([0,T];\mathfrak{D}^\prime (\mathbb{R}^{n-1}))$ locally. 
    
    From Proposition \ref{prop2.1}, we see that \eqref{1.4} is locally equivalent to the following system of equations for $u,v\in C^\infty([0,T];\mathfrak{D}^\prime(\mathbb{R}^{n-1}))$:
    \begin{align*}
        &\Big(\frac{\partial }{\partial x_n} + W\Big)u = v, \quad u|_{x_n=0}=f,\\
        &\Big(\frac{\partial }{\partial x_n} + B - W\Big)v = y \in C^\infty([0,T]\times \mathbb{R}^{n-1}).
    \end{align*}
    Inspired by \cite{LeeUhlm89}, if we substitute $t = T-x_n$ for the second equation above, then we get a backward generalized heat equation
    \begin{align*}
        \frac{\partial v}{\partial t} -(B-W)v=-y.
    \end{align*}
    Since $u$ is smooth in the interior of $ M$ by interior regularity for elliptic operator $\Delta_{\phi}+V$, it follows that $v$ is also smooth in the interior of $ M$, and so $v|_{x_n=T}$ is smooth. In view of the principal symbol $w_1$ of $W$ is positive (see \eqref{2.6}), we get that the solution operator for this heat equation is smooth for $t > 0$ (see \cite[p.\,134]{Treves80}). 
    
    Therefore,
    \begin{align*}
        \frac{\partial u}{\partial x_n} + Wu=v \in C^\infty([0,T]\times \mathbb{R}^{n-1})
    \end{align*}
    locally. If we set $\mathcal{S} f = v|_{\partial  M}$, then $\mathcal{S}$ is a smoothing operator and
    \begin{align*}
        \frac{\partial u}{\partial x_n}\Big|_{\partial  M}=-Wu|_{\partial  M}+\mathcal{S}f.
    \end{align*}
    Combining this and \eqref{1.5}, we obtain
    \begin{align*}
        \Lambda(f) = Wu|_{\partial  M}
    \end{align*}
    modulo a smoothing operator.
\end{proof}

Since the Dirichlet-to-Neumann map $\Lambda$ is defined on the boundary $\partial M$ and it is an elliptic, self-adjoint pseudodifferential operator of order one, it follows from \eqref{2.10} that the full symbol $\sigma(\Lambda)$ of $\Lambda$ has the form
\begin{equation}\label{2.11}
    \sigma(\Lambda)(x^{\prime},\xi^{\prime}) \sim \sum_{j \leqslant 1} w_j(x^{\prime},\xi^{\prime}).
\end{equation}
We apply the methods of Grubb \cite{Grubb86} and Seeley \cite{Seeley67}. Let $S(x^{\prime},\xi^{\prime},\tau)$ be a two-sided parametrix for $\Lambda - \tau$, that is, $S(x^{\prime},\xi^{\prime},\tau)$ is a pseudodifferential operator of order $-1$ in $\xi^{\prime}$ with parameter $\tau$ for which (see \cite[Chap.\,II, Sect.\,9]{Shubin01}, \cite[Chap.\,1, Sect.\,1.6]{Gilkey95} or \cite[p.\,16]{Taylor11.2})
\begin{align*}
    & S(x^{\prime},\xi^{\prime},\tau)(\Lambda - \tau) = 1 \mod OPS^{-\infty}, \\
    & (\Lambda - \tau)S(x^{\prime},\xi^{\prime},\tau) = 1 \mod OPS^{-\infty}.
\end{align*}
Let
\begin{align*}
    s(x^{\prime},\xi^{\prime},\tau) \sim \sum_{j \leqslant -1} s_j(x^{\prime},\xi^{\prime},\tau)
\end{align*}
be the full symbol of $S(x^{\prime},\xi^{\prime},\tau)$. We denote by $s_{j}=s_{j}(x^{\prime},\xi^{\prime},\tau)$, $j\leqslant -1$, for short. Then,
\begin{align}
    \label{2.14}
    &s_{-1} = (w_1-\tau)^{-1},\\
    \label{2.14.1}
    &s_{-1-m} = -s_{-1} \sum_{\substack{-m \leqslant j \leqslant 1 \\ -m \leqslant k \leqslant -1 \\ |J| = j + k + m}} \frac{(-i)^{|J|}}{J !} \partial_{\xi^{\prime}}^{J}  \sigma_j(\Lambda) \, \partial_{x^\prime}^{J}  s_k, \quad m \geqslant 1.
\end{align}

For later reference, we write out the expressions for $s_{-2}$, $s_{-3}$, and $s_{-4}$ as follows:
\begin{align}
    s_{-2} & = -s_{-1}
    \biggl[
        w_0s_{-1} - i\sum_\alpha \frac{\partial w_1}{\partial \xi_\alpha} \frac{\partial s_{-1}}{\partial x_\alpha}
    \biggr], \label{2.16}\\
    s_{-3} & = -s_{-1}
    \bigg[
        w_0s_{-2} + w_{-1}s_{-1} - i \sum_\alpha
        \bigg(
            \frac{\partial w_1}{\partial \xi_\alpha} \frac{\partial s_{-2}}{\partial x_\alpha} + \frac{\partial w_0}{\partial \xi_\alpha} \frac{\partial s_{-1}}{\partial x_\alpha}
        \bigg) \notag\\
        & \quad - \frac{1}{2} \sum_{\alpha,\beta} \frac{\partial^2 w_1}{\partial \xi_{\alpha} \partial \xi_{\beta}} \frac{\partial^2 s_{-1}}{\partial x_{\alpha} \partial x_{\beta}}
    \bigg], \label{2.17}\\
    \label{2.18} s_{-4} & = -s_{-1}
    \bigg[
        w_0s_{-3} + w_{-1}s_{-2} + w_{-2}s_{-1} - i\sum_\alpha \Big(\frac{\partial w_1}{\partial \xi_\alpha} \frac{\partial s_{-3}}{\partial x_\alpha} + \frac{\partial w_0}{\partial \xi_\alpha} \frac{\partial s_{-2}}{\partial x_\alpha}  \notag\\
        & \quad + \frac{\partial w_{-1}}{\partial \xi_\alpha} \frac{\partial s_{-1}}{\partial x_\alpha}\Big)
        - \frac{1}{2} \sum_{\alpha,\beta}\Big(\frac{\partial^2 w_1}{\partial \xi_{\alpha} \partial \xi_{\beta}} \frac{\partial^2 s_{-2}}{\partial x_{\alpha} \partial x_{\beta}} + \frac{\partial^2 w_0}{\partial \xi_{\alpha} \partial \xi_{\beta}} \frac{\partial^2 s_{-1}}{\partial x_{\alpha} \partial x_{\beta}}\Big)  \notag\\
        & \quad + \frac{i}{6} \sum_{\alpha,\beta,\gamma} \frac{\partial^3 w_1}{\partial \xi_{\alpha} \partial \xi_{\beta} \partial \xi_{\gamma}} \frac{\partial^3 s_{-1}}{\partial x_{\alpha} \partial x_{\beta} \partial x_{\gamma}}
    \bigg]. 
\end{align}

\subsection{Coefficients of the asymptotic expansion of heat trace}

Inspired by \cite{LiuTan23,Grubb86,Liu15,Liu19,Seeley67,Seeley69}, we will establish an effective procedure to calculate the first $n-1$ coefficients of the asymptotic expansion of heat trace associated with the Dirichlet-to-Neumann map $\Lambda$.

According to the theory of elliptic equations (see \cite{Morrey66,Morrey58,Stewart74}), we see that the Dirichlet-to-Neumann map $\Lambda$ can generate a strongly continuous generalized heat semigroup $e^{-t\Lambda}$, $t>0$, in a suitable space defined on the boundary $\partial  M$. Furthermore, there exists a Schwartz kernel $\mathcal{K}(t,x^{\prime},y^{\prime})$ such that (see \cite[p.\,4]{Fried64} or \cite{Seeley67})
\begin{align*}
    e^{-t\Lambda} \phi(x^{\prime})
    = \int_{\partial  M} \mathcal{K}(t,x^{\prime},y^{\prime})\phi(y^{\prime}) \,dS(y^{\prime})
\end{align*}
for $\phi \in H^{1}(\partial  M)$. Let $\{u_k\}_{k \geqslant 1}$ be the orthonormal eigenfunctions of the Dirichlet-to-Neumann map $\Lambda$ corresponding to the eigenvalues $\{\lambda_k\}_{k \geqslant 1}$, then the kernel $\mathcal{K}(t,x^{\prime},y^{\prime})$ is given by
\begin{align*}
    \mathcal{K}(t,x^{\prime},y^{\prime})
    = e^{-t\Lambda} \delta(x^{\prime} - y^{\prime})
    = \sum_{k=1}^{\infty} e^{-t\lambda_k} u_k(x^{\prime}) \otimes u_k(y^{\prime}).
\end{align*}
Note that the Dirac function $\delta(x^{\prime} - y^{\prime})$ is defined on $\partial  M\times \partial  M$, one of the entries is fixed. This implies that the integral of the trace of the kernel $\mathcal{K}(t,x^{\prime},y^{\prime})$, i.e.,
\begin{align*}
    \int_{\partial  M} \mathcal{K}(t,x^{\prime},x^{\prime}) \,dS
    = \sum_{k=1}^{\infty} e^{-t\lambda_k}
\end{align*}
is actually a spectral invariant. 

On the other hand, the semigroup $e^{-t\Lambda}$ can also be represented as
\begin{align*}
    e^{-t\Lambda}
    = \frac{i}{2\pi}\int_{\mathcal{C}} e^{-t\tau}(\Lambda-\tau)^{-1} \,d\tau,
\end{align*}
where $\mathcal{C}$ is a suitable curve in the complex plane in the positive direction around the spectrum of $\Lambda$, that is, $\mathcal{C}$ is a contour around the positive real axis. It follows that (see \cite{Agra87,DuisGuill75,GrubbSeeley95,Liu15}), in local coordinates $(x^{\prime},\xi^{\prime})$,
\begin{align*}
    \mathcal{K}(t,x^{\prime},y^{\prime})
    & = e^{-t\Lambda} \delta(x^{\prime} - y^{\prime}) \\
    & = \frac{1}{(2\pi)^{n-1}} \int_{T^{*}_{x^\prime}(\partial  M)} e^{i \langle x^{\prime} - y^{\prime},\xi^{\prime} \rangle} 
    \bigg( 
        \frac{i}{2\pi}\int_{\mathcal{C}} e^{-t\tau} \sigma((\Lambda-\tau)^{-1}) \,d\tau
    \bigg) \,d\xi^{\prime} \\
    & = \frac{1}{(2\pi)^{n-1}} \int_{T^{*}_{x^\prime}(\partial  M)} e^{i \langle x^{\prime} - y^{\prime},\xi^{\prime} \rangle}
    \bigg(
        \frac{i}{2\pi}\int_{\mathcal{C}} e^{-t\tau} \sum_{j \leqslant -1} s_{j}(x^{\prime},\xi^{\prime},\tau) \,d\tau 
    \bigg) \,d\xi^{\prime},
\end{align*}
where $\sigma((\Lambda-\tau)^{-1})\sim\sum_{j \leqslant -1} s_{j}(x^{\prime},\xi^{\prime},\tau)$ is the full symbol of the pseudodifferential operator $(\Lambda-\tau)^{-1}$, such that the restriction to the diagonal of the kernel $\mathcal{K}(t,x^{\prime},y^{\prime})$ is
\begin{align*}
    \mathcal{K}(t,x^{\prime},x^{\prime})
    = \frac{1}{(2\pi)^{n-1}} \int_{T^{*}_{x^\prime}(\partial  M)}
    \bigg(
        \frac{i}{2\pi} \int_{\mathcal{C}} e^{-t\tau}\sum_{j \leqslant -1} s_{j}(x^{\prime},\xi^{\prime},\tau) \,d\tau 
    \bigg) \,d\xi^{\prime}.
\end{align*}
Therefore,
\begin{align*}
    \sum_{k=1}^{\infty} e^{-t\lambda_k}
    = \int_{\partial  M} 
    \bigg[
        \frac{1}{(2\pi)^{n-1}} \int_{T^{*}_{x^\prime}(\partial  M)}
        \bigg(
            \frac{i}{2\pi}\int_{\mathcal{C}} e^{-t\tau}\sum_{j \leqslant -1} s_{j}(x^{\prime},\xi^{\prime},\tau) \,d\tau 
        \bigg) \,d\xi^{\prime}
    \bigg] \,dS.
\end{align*}
We will calculate the asymptotic expansion of the trace of the semigroup $e^{-t\Lambda}$ as $t \to 0^{+}$. More precisely, we will establish a procedure to calculate
\begin{equation}\label{3.1}
    a_{k}(x^{\prime}) 
    = \frac{i}{(2 \pi)^{n}}\int_{T^{*}_{x^\prime}(\partial  M)}\int_{\mathcal{C}} e^{-\tau}s_{-1-k}(x^{\prime},\xi^{\prime},\tau) \,d\tau \,d\xi^{\prime}
\end{equation}
for $0 \leqslant k \leqslant n-1$ (see \cite{Seeley67} and also \cite{LiuTan23,Liu15,Liu19,PoltSher15}).

\section{Computations of \texorpdfstring{$a_0(x^\prime)$ and $a_1(x^\prime)$}{}}\label{s4}

In what follows, we will give the explicit expressions for the first four coefficients of the asymptotic expansion of heat trace in boundary normal coordinates. We compute these coefficients at the origin $x_0$ in boundary normal coordinates since the origin $x_0\in\partial M$ can be chosen arbitrarily. Hence, the coefficients can be obtained immediately for any $x^{\prime}\in\partial M$.

It is convenient to calculate the coefficients of the asymptotic expansion by the following two lemmas.
\begin{lemma}\label{lem4.1}
    For any $k \geqslant 1$, we have
    \begin{equation*}
        \frac{1}{2\pi i} \int_{\mathcal{C}} s_{-1}^k e^{-\tau} \,d\tau 
        = -\frac{1}{(k-1)!}e^{-w_1}.
    \end{equation*}
\end{lemma}

\begin{proof}
    Since $s_{-1} = (w_1-\tau)^{-1}$ by \eqref{2.14}. Then the proof is a simple computation with calculus of residues.
\end{proof}

\begin{lemma}\label{lem4.2}
    For any $k$, we have the following integral formulas:
    \begin{align*}
        \int_{\mathbb{R}^{n-1}} w_1^{k} e^{-w_1} \xi_{\alpha}^m \,d\xi^{\prime} &=
        \begin{cases}
            \omega_{n-2}\Gamma(n+k-1), \quad & m = 0, \\
            0, \quad & m = 1,
        \end{cases}\\[2mm]
        \int_{\mathbb{R}^{n-1}} w_1^{k-2} e^{-w_1} \xi_{\alpha} \xi_{\beta} \,d\xi^{\prime} &=
        \begin{cases}
            \displaystyle \frac{\omega_{n-2}\Gamma(n+k-1)}{n-1} , \quad & \alpha = \beta, \\
            0, \quad & \alpha \neq \beta,
        \end{cases}\\[2mm]
        \int_{\mathbb{R}^{n-1}} w_1^{k-4} e^{-w_1} \xi_{\alpha}^2 \xi_{\beta}^2 \,d\xi^{\prime} &=
        \begin{cases}
            \displaystyle \frac{3\omega_{n-2}\Gamma(n+k-1)}{n^{2}-1} , \quad & \alpha = \beta, \\[3mm]
            \displaystyle \frac{\omega_{n-2}\Gamma(n+k-1)}{n^{2}-1} , \quad & \alpha \neq \beta.
        \end{cases}
    \end{align*}
\end{lemma}

\begin{proof}
    The proof can be obtained by applying the spherical coordinates transform (see also \cite{Liu15}).
\end{proof}

In order to simplify the computations and to investigate the important roles of the drifting function $\phi$ and the potential $V$, we consider the terms that only involve $\phi$ or $V$ in an expression $f(x)$ (or $f(x,\xi^\prime,\tau)$) at the origin $x_{0}$ in boundary normal coordinates, and we denote by $f(\phi,V)$ the sum of these terms that only involve $\phi$ or $V$ at $x_{0}$. Thus, it follows from formula \eqref{3.1} that
\begin{equation}\label{3.9}
    a_{k}(\phi,V) 
    = \frac{i}{(2 \pi)^{n}}\int_{\mathbb{R}^{n-1}}\int_{\mathcal{C}} e^{-\tau}s_{-1-k}(\phi,V) \,d\tau \,d\xi^{\prime}
\end{equation}
for $0 \leqslant k \leqslant n-1$. Therefore, we get
\begin{align}\label{3.10}
    a_{k}(x_0) = \tilde{a}_{k}(x_0) + a_{k}(\phi,V)
\end{align}
for $0 \leqslant k \leqslant n-1$, where $\tilde{a}_{k}(x_0)$ are the spectral invariants of the classical Dirichlet-to-Neumann map associated with the Laplace--Beltrami operator.

For later reference, here we list the expressions for $\tilde{a}_k(x^\prime)$, $k=0,1,2,3$, as follows (see \cite[Theorem 5.1]{Liu15}). Noting that the dimension of the manifold is $n+1$ in \cite{Liu15}, here we have changed the dimension to $\dim M = n$.
\begin{align}
    \tilde{a}_0(x^{\prime}) & = \frac{\omega_{n-2}\Gamma(n-1)}{(2\pi)^{n-1}}, \quad n \geqslant 1, \label{328}\\
    \tilde{a}_1(x^{\prime}) & = \frac{(n-2)\omega_{n-2}\Gamma(n-1)}{2(2\pi)^{n-1}(n-1)}H, \quad n \geqslant 2, \label{329}\\
    \tilde{a}_2(x^{\prime}) & = \frac{\omega_{n-2}\Gamma(n-2)}{4(2\pi)^{n-1}}
    \biggl[
        \frac{n^3-4n^2+n+8}{2(n^2-1)}H^2 + \frac{n(n-3)}{2(n^2-1)} \sum_{\alpha=1}^{n-1} \kappa_{\alpha}^2  \notag\\
        & \quad + \frac{n-2}{2(n-1)}\tilde{R} - \frac{n-4}{6(n-1)}R
    \biggr], \quad n \geqslant 3, \label{330}\\
    \tilde{a}_3(x^{\prime}) & = \frac{\omega_{n-2}\Gamma(n-3)}{8(2\pi)^{n-1}}
    \biggl[
        \frac{n^5-5n^4-10n^3+52n^2+2n-114}{6(n^2-1)(n+3)}H^3 \notag\\
        & \quad + \frac{n^4-n^3-12n^2+22n+6}{2(n^2-1)(n+3)}H\sum_{\alpha=1}^{n-1} \kappa_{\alpha}^2 - \frac{4(n-3)(n-2)}{3(n^2-1)(n+3)}\sum_{\alpha=1}^{n-1} \kappa_{\alpha}^3  \notag\\
        & \quad + \frac{n^3-6n^2+2n+14}{2(n^2-1)}H \tilde{R} - \frac{3n^3-20n^2+12n+42}{6(n^2-1)}HR \notag\\
        & \quad + \frac{2(n^2-3n+1)}{n^2-1}\sum_{\alpha=1}^{n-1} \kappa_{\alpha}\tilde{R}_{\alpha\alpha} - \frac{2n(3n-8)}{3(n^2-1)}\sum_{\alpha=1}^{n-1} \kappa_{\alpha}R_{\alpha\alpha} \notag\\
        & \quad + \frac{n-2}{n-1} \nabla_{\frac{\partial}{\partial x_n}}\tilde{R}_{nn}
    \biggr], \quad n \geqslant 4. \label{331}
\end{align}

Now we compute the first two coefficients $a_0(x^\prime)$ and $a_1(x^\prime)$.

\subsection{Computation of \texorpdfstring{$a_0(x^\prime)$}{}}

Since $s_{-1}(\phi,V)=0$ by \eqref{2.6} and \eqref{2.14}. Applying formula \eqref{3.1}, we get
\begin{align}\label{a00}
    a_{0}(\phi,V)=0.
\end{align}
Therefore, combining \eqref{3.10}, \eqref{328}, and \eqref{a00}, we immediately obtain \eqref{a0}.

\subsection{Computation of \texorpdfstring{$a_1(x^\prime)$}{}}

It follows from \eqref{2.6} and \eqref{10.1} that
\begin{align}\label{8.2}
    \frac{\partial w_1^m}{\partial x_j}
    =\frac{m}{2}w_1^{m-2}\frac{\partial w_1^2}{\partial x_j}
    =\frac{m}{2}w_1^{m-2} \sum_{\alpha,\beta}g^{\alpha\beta,j}\xi_\alpha\xi_{\beta},\quad m\in\mathbb{Z}.
\end{align}
We denote by $\frac{\partial w_1^m}{\partial x_{j}}(x_0)=\frac{\partial w_1^m}{\partial x_{j}}(x_0,\xi^{\prime})$, $1 \leqslant j \leqslant n$. Then, at the origin $x_0$ in boundary coordinates, by applying \eqref{6.3}--\eqref{6.3.2} we get $\xi^\alpha(x_0)=\xi_\alpha$ and
\begin{align}\label{8.7}
    \frac{\partial w_1^m}{\partial x_{\alpha}}(x_0)=0,\quad 
    \frac{\partial w_1^m}{\partial x_{n}}(x_0)=m w_1^{m-2}\sum_\alpha\kappa_{\alpha} \xi_{\alpha}^2,\quad m\in\mathbb{Z}.
\end{align}

By combining \eqref{6.1}--\eqref{6.3.2}, \eqref{6.6}, and \eqref{6.8}, we have
\begin{align}
    \label{3.12.1} b(x_{0})& =-(H+\phi_n),\\
    \label{3.12.2} c_1(x_{0}) &= -i\sum_\alpha \phi_\alpha\xi_\alpha,
\end{align}
where $c_1(x_0)=c_1(x_0,\xi^{\prime})$.

From \eqref{2.14}, \eqref{2.16}, and \eqref{8.7} we compute
\begin{align}
    \label{8.6}
    &\frac{\partial s_{-1}}{\partial x_\alpha} = -s_{-1}^2\frac{\partial w_1}{\partial x_\alpha},\quad
    \frac{\partial s_{-1}}{\partial x_\alpha}(x_0)=0,\\
    \label{8.4}
    &s_{-2}(x_{0}) = -s_{-1}^2w_0.
\end{align}

Note that, for any function $f(\xi^{\prime})$, any term that is odd in $\xi_\alpha$ for any particular index $\alpha$ will integrate to zero when integrating over $\mathbb{R}^{n-1}$. Thus, we only consider the terms that are even in $\xi_\alpha$ in $f(\xi^{\prime})$ when integrating over $\mathbb{R}^{n-1}$. We henceforth denote by (cf. \cite{PoltSher15})
    \begin{align*}
        f(\xi^{\prime}) \cong  g(\xi^{\prime}),
    \end{align*}
    where $g(\xi^{\prime})$ is the sum of the terms that are even in $\xi_\alpha$ in $f(\xi^{\prime})$. More precisely, $f(\xi^{\prime}) \cong g(\xi^{\prime})$ if and only if
    \begin{equation*}
        \int_{\mathbb{R}^{n-1}} e^{-w_1}\big[f(\xi^{\prime}) - g(\xi^{\prime})\big] \,d\xi^{\prime} = 0.
    \end{equation*}
    For example,
    \begin{align*}
        \biggl(\sum_{\alpha}\phi_{\alpha}\xi_\alpha\biggr)^2+\sum_{\alpha}\kappa_{\alpha}\xi_\alpha
        =\sum_{\alpha,\beta}\phi_{\alpha}\phi_\beta\xi_\alpha\xi_\beta+\sum_{\alpha}\kappa_{\alpha}\xi_\alpha \cong \sum_{\alpha}\phi_{\alpha}^2\xi_\alpha^2.
    \end{align*}

Substituting \eqref{6.6} and \eqref{6.8} into \eqref{2.7}, we have
\begin{align}\label{8.1}
    w_0(\phi,V) \cong -\frac{1}{2}\phi_n.
\end{align}
Then, from \eqref{8.4} and \eqref{8.1}, we obtain
\begin{align}\label{8.5}
    s_{-2}(\phi,V)\cong \frac{1}{2}s_{-1}^2\phi_n.
\end{align}
It follows from formula \eqref{3.1}, Lemma \ref{lem4.1} and \ref{lem4.2} that
\begin{align}\label{a11}
    a_{1}(\phi,V)
    = \frac{\omega_{n-2}\Gamma(n-1)}{2(2\pi)^{n-1}}\phi_n.
\end{align}
Therefore, combining \eqref{3.10}, \eqref{329}, and \eqref{a11}, we immediately obtain \eqref{a1}.

\section{Computation of \texorpdfstring{$a_2(x^\prime)$}{}}\label{s5}

Recall that $\frac{\partial s_{-1}}{\partial x_\alpha}(x_0)=0$ by \eqref{8.6}, $w_1$ and $s_{-1}$ do not involve $\phi$ or $V$. Thus, from \eqref{2.17} we see that
\begin{align}\label{8.8}
    s_{-3}(\phi,V) 
    = -s_{-1} \biggl(w_0s_{-2} + w_{-1}s_{-1} - i \sum_\alpha\frac{\partial w_1}{\partial \xi_\alpha} \frac{\partial s_{-2}}{\partial x_\alpha}\biggr).
\end{align}
Note that $\frac{\partial w_{1}}{\partial x_\alpha}(x_0)=0$ by \eqref{8.7}. From \eqref{2.8} we have
\begin{align}\label{8.10}
    w_{-1}(\phi,V) 
    = \frac{1}{2} w_1^{-1} 
    \bigg(
        -w_0^2 + i\sum_\alpha \frac{\partial w_1}{\partial \xi_\alpha}\frac{\partial w_0}{\partial x_\alpha} + bw_0 + \frac{\partial w_0}{\partial x_n} - c_0
    \bigg).
\end{align}
By \eqref{2.16} we have
\begin{align}\label{b4}
    \frac{\partial s_{-2}}{\partial x_{\alpha}}(\phi,V)=-s_{-1}^2\frac{\partial w_{0}}{\partial x_{\alpha}}.
\end{align}
Then, substituting \eqref{8.4}, \eqref{8.10}, and \eqref{b4} into \eqref{8.8} we get 
\begin{align}\label{b9}
    s_{-3}(\phi,V)
    &=\Big(s_{-1}^3+\frac{1}{2}s_{-1}^2w_1^{-1}\Big)\Big(w_0^2-i\sum_{\alpha}\frac{\partial w_1}{\partial \xi_\alpha}\frac{\partial w_0}{\partial x_\alpha}\Big) -\frac{1}{2}s_{-1}^2w_1^{-1}\Big(bw_0+\frac{\partial w_0}{\partial x_n}-c_0\Big).
\end{align}

From \eqref{2.7} we have
\begin{align}\label{b5}
    w_0(x_0)=\frac{1}{2}\Big(b+w_{1}^{-1}\frac{\partial w_1}{\partial x_n}-w_1^{-1}c_1\Big).
\end{align}
Then,
\begin{align}
    \label{b6} &w_0^2(\phi,V)\cong\frac{1}{4}b^2+\frac{1}{2}bw_{1}^{-1}\frac{\partial w_1}{\partial x_n}+\frac{1}{4}w_1^{-2}c_1^2,\\
    &\frac{\partial w_0}{\partial x_\alpha}(\phi,V)=\frac{1}{2}w_1^{-1}\Big(\frac{\partial b}{\partial x_\alpha}-\frac{\partial c_{1}}{\partial x_\alpha}\Big).\notag
\end{align}
By \eqref{2.6} we see that $w_1=|\xi^\prime|$ is even in $\xi^{\prime}$ and without $\phi$ or $V$. Thus $\frac{\partial w_1}{\partial \xi_\alpha}$ is odd in $\xi^{\prime}$, whereas $c_1$ and $\frac{\partial c_{1}}{\partial x_\alpha}$ are odd in $\xi^{\prime}$. Hence,
\begin{align}\label{b7}
    \Big(\sum_{\alpha}\frac{\partial w_1}{\partial \xi_\alpha}\frac{\partial w_0}{\partial x_\alpha}\Big)(\phi,V)\cong -\frac{1}{2}w_1^{-1}\sum_{\alpha}\frac{\partial w_1}{\partial \xi_\alpha}\frac{\partial c_1}{\partial x_\alpha}.
\end{align}
By \eqref{2.7} we have
\begin{align}\label{b8}
    \frac{\partial w_0}{\partial x_n}(\phi,V)\cong \frac{1}{2}\frac{\partial b}{\partial x_n}.
\end{align}
Thus, substituting \eqref{b5}, \eqref{b6}, \eqref{b7}, and \eqref{b8} into \eqref{b9} we get 
\begin{align}\label{b12}
    s_{-3}(\phi,V) &\cong \frac{1}{4}\Big(s_{-1}^3-\frac{1}{2}s_{-1}^2w_1^{-1}\Big)b^2+\frac{1}{2}s_{-1}^3w_1^{-1}b\frac{\partial w_1}{\partial x_n}+\frac{1}{4}\Big(s_{-1}^3w_1^{-2}+\frac{1}{2}s_{-1}^2w_1^{-3}\Big)c_1^2 \notag\\
    &\quad +\frac{i}{2}\Big(s_{-1}^3w_1^{-1}+\frac{1}{2}s_{-1}^2w_1^{-2}\Big)\sum_{\alpha}\frac{\partial w_1}{\partial \xi_\alpha}\frac{\partial c_1}{\partial x_\alpha}-\frac{1}{2}s_{-1}^2w_1^{-1}\Big(\frac{1}{2}\frac{\partial b}{\partial x_n}-c_0\Big).
\end{align}

It follows from \eqref{8.7} and \eqref{3.12.1} that
\begin{align}
    &b^2(\phi,V)\cong \phi_n^2+2H\phi_n,\label{b10}\\
    &\frac{\partial w_1}{\partial x_n}(x_0)=w_1^{-1}\sum_{\alpha}\kappa_{\alpha}\xi_{\alpha}^2,\\
    &\Big(b\frac{\partial w_1}{\partial x_n}\Big)(\phi,V)=-\phi_n w_1^{-1}\sum_{\alpha}\kappa_{\alpha}\xi_{\alpha}^2,\\
    &\frac{\partial b}{\partial x_n}(\phi,V)=-\phi_{nn}.
\end{align}
Recall that $w_1=|\xi^\prime|$ by \eqref{2.6}. Then, we compute
\begin{align}\label{8.15}
    \frac{\partial w_1^m}{\partial \xi_\alpha}
    =m w_1^{m-1} \frac{\partial w_1}{\partial \xi_\alpha}
    =m w_1^{m-2}\xi^{\alpha},\quad
    \frac{\partial w_1^m}{\partial \xi_{\alpha}}(x_0)
    =m w_1^{m-2}\xi_{\alpha},\quad m\in\mathbb{Z}.
\end{align}
By \eqref{3.12.2}, \eqref{6.8}, and \eqref{8.15} we see that
\begin{align}
    &c_1^2(\phi,V)\cong -\sum_{\alpha}\phi_{\alpha}^2\xi_{\alpha}^2,\\
    &\frac{\partial c_1}{\partial x_\alpha}(\phi,V)=-i\sum_{\beta}\phi_{\alpha\beta}^2\xi_{\beta},\\
    &\Big(\sum_{\alpha}\frac{\partial w_1}{\partial \xi_\alpha}\frac{\partial c_1}{\partial x_\alpha}\Big)(\phi,V)\cong -iw_1^{-1}\sum_{\alpha}\phi_{\alpha\alpha}\xi_{\alpha}^2.\label{b11}
\end{align}
Thus, substituting \eqref{b10}--\eqref{b11}, and \eqref{6.9} into \eqref{b12} we obtain
\begin{align}
    s_{-3}(\phi,V) &\cong \frac{1}{4}\Big(s_{-1}^3-\frac{1}{2}s_{-1}^2w_1^{-1}\Big)(\phi_n^2+2H\phi_n)-\frac{1}{2}s_{-1}^3w_1^{-2}\phi_n \sum_{\alpha}\kappa_{\alpha}\xi_{\alpha}^2\notag\\
    &\quad -\frac{1}{4}\Big(s_{-1}^3w_1^{-2}+\frac{1}{2}s_{-1}^2w_1^{-3}\Big)\sum_{\alpha}\phi_{\alpha}^2\xi_{\alpha}^2 +\frac{1}{2}\Big(s_{-1}^3w_1^{-2}+\frac{1}{2}s_{-1}^2w_1^{-3}\Big)\sum_{\alpha}\phi_{\alpha\alpha}\xi_{\alpha}^2\notag\\
    &\quad +\frac{1}{2}s_{-1}^2w_1^{-1}\Big(\frac{1}{2}\phi_{nn}+V\Big).
\end{align}

By applying formula \eqref{3.9}, Lemma \ref{lem4.1} and \ref{lem4.2}, we compute
\begin{align}\label{8.31}
    a_2(\phi,V) & = \frac{\omega_{n-2}\Gamma(n-2)}{4(2\pi)^{n-1}}
        \biggl[
            \frac{n-2}{2}\phi_{n}^2 + \frac{n^2-5n+5}{n-1}H\phi_{n} + \Delta\phi - \frac{1}{2}|\nabla\phi|^2 + 2V 
        \biggr].
\end{align}
Therefore, combining \eqref{3.10}, \eqref{330}, and \eqref{8.31} we immediately obtain \eqref{a2}.

\section{Computation of \texorpdfstring{$a_3(x^{\prime})$}{}}\label{s6}

It follows from \eqref{2.6}, \eqref{10.1}, and \eqref{8.2} that
\begin{align}\label{7.11}
    \frac{\partial^2 w_1}{\partial x_j \partial x_k}
    &=\frac{1}{2}w_1^{-1}\sum_{\alpha,\beta} g^{\alpha\beta,jk}\xi_{\alpha}\xi_\beta-\frac{1}{4}w_1^{-3}\biggl(\sum_{\alpha,\beta}g^{\alpha\beta,j}\xi_\alpha\xi_{\beta}\biggr)\biggl(\sum_{\alpha,\beta}g^{\alpha\beta,k}\xi_\alpha\xi_{\beta}\biggr).
\end{align}
In particular, by \eqref{6.3}--\eqref{6.3.2} and \eqref{7.11}, we get
\begin{align}
    \label{7.1.2} &\frac{\partial^2 w_1}{\partial x_\alpha \partial x_\beta}(x_0)=\frac{1}{2}w_1^{-1} \sum_{\gamma,\rho} g^{\gamma\rho,\alpha\beta}\xi_{\gamma}\xi_\rho,\\
    \label{7.1.3} &\frac{\partial^2 w_1}{\partial x_n \partial x_\alpha}(x_0)=0, \quad 
    \frac{\partial^2 w_1^{-1}}{\partial x_n \partial x_\alpha}(x_0)=0,\\
    \label{7.1.4} &\frac{\partial^2 w_1}{\partial x_n^2}(x_0)=\frac{1}{2}w_1^{-1}\sum_{\alpha,\beta} g^{\alpha\beta,nn}\xi_{\alpha}\xi_\beta-w_1^{-3}\sum_{\alpha,\beta}\kappa_\alpha\kappa_\beta\xi_\alpha^2\xi_\beta^2.
\end{align}

\begin{lemma}
   In boundary normal coordinates, at the origin $x_0\in\partial M$,
    \begin{align}
        &\sum_{\alpha,\beta} \frac{\partial w_1}{\partial \xi_\alpha}\frac{\partial w_1}{\partial \xi_\beta}\frac{\partial^2 w_1}{\partial x_{\alpha} \partial x_{\beta}}(x_0)
        \cong 0,\label{16.6}\\
        &\biggl(\sum_{\alpha,\beta} \frac{\partial w_1}{\partial \xi_\alpha}\frac{\partial w_0}{\partial \xi_\beta}\frac{\partial^2 w_1}{\partial x_{\alpha} \partial x_{\beta}}\biggr)(\phi,V)
        \cong 0,\label{16.7}\\
        &\biggl(\sum_{\alpha,\beta}\frac{\partial^2 w_0}{\partial \xi_{\alpha} \partial \xi_{\beta}}\frac{\partial^2 w_1}{\partial x_\alpha\partial x_\beta}\biggr)(\phi,V)
        \cong 0.\label{16.8}
    \end{align}
\end{lemma}

\begin{proof}
    It follows from \eqref{8.15} and \eqref{7.1.2} that
    \begin{align}\label{a4}
        \sum_{\alpha,\beta} \frac{\partial w_1}{\partial \xi_\alpha}\frac{\partial w_1}{\partial \xi_\beta}\frac{\partial^2 w_1}{\partial x_{\alpha} \partial x_{\beta}}(x_0)
        =\frac{1}{2}w_1^{-3} \sum_{\alpha,\beta,\gamma,\rho} g^{\gamma\rho,\alpha\beta}\xi_\alpha\xi_\beta\xi_{\gamma}\xi_\rho.
    \end{align}
    Note that
\begin{align}\label{8.85}
    \sum_{\alpha,\beta,\gamma,\rho} g^{\gamma\rho,\alpha\beta}\xi_\alpha\xi_\beta\xi_\gamma\xi_\rho \notag 
    &\cong \biggl(\sum_{\alpha=\beta=\gamma=\rho}+\sum_{\alpha=\beta\neq\gamma=\rho}+\sum_{\alpha=\gamma\neq\beta=\rho}+\sum_{\alpha=\rho\neq\beta=\gamma}\biggr)g^{\gamma\rho,\alpha\beta}\xi_\alpha\xi_\beta\xi_\gamma\xi_\rho \notag\\
    &=\sum_\alpha g^{\alpha\alpha,\alpha\alpha}\xi_\alpha^4 + \sum_{\alpha\neq\gamma}g^{\gamma\gamma,\alpha\alpha}\xi_\alpha^2\xi_\gamma^2 +\sum_{\alpha\neq\beta}g^{\alpha\beta,\alpha\beta}\xi_\alpha^2\xi_\beta^2 +\sum_{\alpha\neq\beta}g^{\beta\alpha,\alpha\beta}\xi_\alpha^2\xi_\beta^2 \notag\\
    &=\sum_\alpha g^{\alpha\alpha,\alpha\alpha}\xi_\alpha^4 + \sum_{\alpha\neq\beta}(g^{\beta\beta,\alpha\alpha}+2g^{\alpha\beta,\alpha\beta})\xi_\alpha^2\xi_\beta^2. 
\end{align}
From formula $(3.51)$ in \cite[p.\,555]{Taylor11.2}, we have
\begin{align}\label{7.1.5}
    g_{\alpha\beta,\gamma\rho}(x_0)=g_{\gamma\rho,\alpha\beta}(x_0)=\frac{1}{3}(R_{\alpha\gamma\beta\rho}+R_{\alpha\rho\beta\gamma})(x_0).
\end{align}
Note also that $\sum_k g_{jk}g^{kl}=\delta_{jl}$. We have
\begin{align*}
    0=\sum_k \frac{\partial^2}{\partial x_m\partial x_\alpha} (g_{jk}g^{kl})
    =\sum_k(g_{jk,m\alpha}g^{kl}+g_{jk,\alpha}g^{kl,m}+g_{jk,m}g^{kl,\alpha}+g_{jk}g^{kl,m\alpha}).
\end{align*}
Combining \eqref{6.3}, \eqref{6.3.1} and the above equality, we obtain
\begin{align}
    g_{jl,m\alpha}(x_0)= -g^{jl,m\alpha}(x_0).
\end{align}
In particular,
\begin{align}
    \label{7.1.1} g^{\alpha\beta,n\gamma}(x_0)&=-g_{\alpha\beta,n\gamma}(x_0)=0,\\
    \label{7.1.1.1} g^{\alpha\beta,\gamma\rho}(x_0)&=-g_{\alpha\beta,\gamma\rho}(x_0).
\end{align}
It follows from \eqref{7.1.5}, \eqref{7.1.1.1}, and \eqref{rm} that, at $x_0$,
\begin{align}
    \label{8.86} g^{\alpha\alpha,\alpha\alpha}(x_0)
    &=-\frac{2}{3}R_{\alpha\alpha\alpha\alpha}=0,\\
    (g^{\beta\beta,\alpha\alpha}+2g^{\alpha\beta,\alpha\beta})(x_0)
    \label{8.86.1} &=-\frac{1}{3}(2R_{\beta\alpha\beta\alpha}+2R_{\alpha\alpha\beta\beta}+2R_{\alpha\beta\beta\alpha})
    =0.
\end{align}
By \eqref{8.85}, \eqref{8.86}, and \eqref{8.86.1}, we conclude that (see also (6.56) in \cite{LiuTan23})
\begin{align}\label{8.87}
    \sum_{\alpha,\beta,\gamma,\rho} g^{\gamma\rho,\alpha\beta}(x_0)\xi_\alpha\xi_\beta\xi_\gamma\xi_\rho
    \cong 0.
\end{align}
Combining \eqref{a4} and \eqref{8.87} we immediately get \eqref{16.6}.

By \eqref{2.6} we see that $w_1=|\xi^\prime|$ is even in $\xi^{\prime}$. Thus, $\frac{\partial w_1}{\partial \xi_{\alpha}}$ is odd in $\xi^{\prime}$ and $\frac{\partial^2 w_1}{\partial x_{\alpha} \partial x_{\beta}}$ is even in $\xi^{\prime}$. From \eqref{2.7}, we know that $\frac{\partial w_0}{\partial \xi_\alpha}(\phi,V)$ is even in $\xi^{\prime}$, whereas $\frac{\partial^2 w_0}{\partial \xi_{\alpha} \partial \xi_{\beta}}(\phi,V)$ is odd in $\xi^{\prime}$. This implies \eqref{16.7} and \eqref{16.8}.
\end{proof}

Recall that $\frac{\partial s_{-1}}{\partial x_\alpha}(x_0)=0$ by \eqref{8.6}, and $w_1$ and $s_{-1}$ do not involve $\phi$ or $V$. It follows from \eqref{2.18} and \eqref{16.8} that
\begin{align}\label{8.32}
    s_{-4}(\phi,V) & = -s_{-1}
    \bigg[
        w_0s_{-3} + w_{-1}s_{-2} + w_{-2}s_{-1} - i\sum_\alpha \Big(\frac{\partial w_1}{\partial \xi_\alpha} \frac{\partial s_{-3}}{\partial x_\alpha} + \frac{\partial w_0}{\partial \xi_\alpha} \frac{\partial s_{-2}}{\partial x_\alpha} \Big) \notag\\
        & \quad - \frac{1}{2} \sum_{\alpha,\beta}\frac{\partial^2 w_1}{\partial \xi_{\alpha} \partial \xi_{\beta}} \frac{\partial^2 s_{-2}}{\partial x_{\alpha} \partial x_{\beta}}
    \bigg]. 
\end{align}
By \eqref{2.17} and \eqref{2.8} we have
\begin{align}\label{a7}
    s_{-3}(x_0) 
    = -s_{-1} \biggl(w_0s_{-2} + w_{-1}s_{-1} - i \sum_\alpha\frac{\partial w_1}{\partial \xi_\alpha} \frac{\partial s_{-2}}{\partial x_\alpha} -\frac{1}{2} \sum_{\alpha,\beta} \frac{\partial^2 w_1}{\partial \xi_{\alpha} \partial \xi_{\beta}} \frac{\partial^2 s_{-1}}{\partial x_{\alpha} \partial x_{\beta}}\biggr),
\end{align}
and
\begin{align}\label{a8}
    w_{-1}(x_0) 
        & = \frac{1}{2} w_1^{-1} 
        \bigg[
            -w_0^2 + i\sum_\alpha \frac{\partial w_1}{\partial \xi_\alpha}\frac{\partial w_0}{\partial x_\alpha}  + \frac{1}{2}\sum_{\alpha,\beta} \frac{\partial^2 w_1}{\partial \xi_{\alpha} \partial \xi_{\beta}}\frac{\partial^2 w_1}{\partial x_{\alpha} \partial x_{\beta}} \notag\\
            & \quad + bw_0 + \frac{\partial w_0}{\partial x_n} - c_0 
        \bigg].
\end{align}
From \eqref{2.14} and \eqref{2.16} we get
\begin{align}
    \label{a5} &\frac{\partial s_{-2}}{\partial x_\alpha}(x_0)
    = - s_{-1}^2\frac{\partial w_0}{\partial x_\alpha} - is_{-1}^3 \sum_\beta \frac{\partial w_1}{\partial \xi_\beta}\frac{\partial^2 w_1}{\partial x_{\alpha} \partial x_{\beta}},\\
    \label{a6} &\frac{\partial^2 s_{-1}}{\partial x_{\alpha} \partial x_{\beta}}(x_0)
    =-s_{-1}^2\frac{\partial^2 w_1}{\partial x_{\alpha} \partial x_{\beta}}.
\end{align}
Substituting \eqref{8.4}, \eqref{a5}, \eqref{a6}, and \eqref{a8} into \eqref{a7}, and combining \eqref{16.6}, we obtain
\begin{align}\label{a15}
    w_0s_{-3}(x_0)
    &\cong \Big(s_{-1}^3 + \frac{1}{2}s_{-1}^2w_1^{-1}\Big)w_0^3-i\Big(s_{-1}^3 + \frac{1}{2}s_{-1}^2w_1^{-1}\Big)w_0\sum_\alpha \frac{\partial w_1}{\partial \xi_\alpha} \frac{\partial w_0}{\partial x_\alpha} \notag\\
    & \quad -\Big(\frac{1}{2}s_{-1}^3 + \frac{1}{4}s_{-1}^2w_1^{-1}\Big)w_0\sum_{\alpha,\beta} \frac{\partial^2 w_1}{\partial \xi_{\alpha} \partial \xi_{\beta}} \frac{\partial^2 w_1}{\partial x_{\alpha} \partial x_{\beta}} \notag\\
    & \quad -\frac{1}{2}s_{-1}^2w_1^{-1}\Big(bw_0^2 + w_0\frac{\partial w_0}{\partial x_n} - w_0c_0 \Big).
\end{align}

Combining \eqref{8.4} and \eqref{a8} we get
\begin{align}\label{a16}
    w_{-1}s_{-2}(x_0)
    &=\frac{1}{2}s_{-1}^2w_1^{-1}w_0^3-\frac{i}{2}s_{-1}^2w_1^{-1}w_0\sum_\alpha \frac{\partial w_1}{\partial \xi_\alpha} \frac{\partial w_0}{\partial x_\alpha}-\frac{1}{4}s_{-1}^2w_1^{-1}w_0\sum_{\alpha,\beta} \frac{\partial^2 w_1}{\partial \xi_{\alpha} \partial \xi_{\beta}} \frac{\partial^2 w_1}{\partial x_{\alpha} \partial x_{\beta}} \notag\\
    & \quad -\frac{1}{2}s_{-1}^2w_1^{-1}\Big(bw_0^2 + w_0\frac{\partial w_0}{\partial x_n} - w_0c_0 \Big).
\end{align}

By \eqref{2.9} and \eqref{2.8} we see that
\begin{align}
    \label{a9}w_{-2}(\phi,V)
    &= \frac{1}{2} w_1^{-1} 
    \bigg[
        -2w_0w_{-1} + i\sum_\alpha \Big(\frac{\partial w_1}{\partial \xi_\alpha}\frac{\partial w_{-1}}{\partial x_\alpha}+\frac{\partial w_0}{\partial \xi_\alpha}\frac{\partial w_0}{\partial x_\alpha}  \Big) \notag\\
        & \quad + \frac{1}{2} \sum_{\alpha,\beta} \frac{\partial^2 w_1}{\partial \xi_{\alpha} \partial \xi_{\beta}} \frac{\partial^2 w_0}{\partial x_{\alpha} \partial x_{\beta}}  + bw_{-1} + \frac{\partial w_{-1}}{\partial x_n}
    \bigg],\\
    \label{a10}\frac{\partial w_{-1}}{\partial x_\alpha}(\phi,V)
    &= \frac{1}{2} w_1^{-1}
    \bigg[
        - 2w_0\frac{\partial w_0}{\partial x_\alpha}  + i\sum_\beta \Big(\frac{\partial w_1}{\partial \xi_\beta}\frac{\partial^2 w_0}{\partial x_\alpha \partial x_\beta} + \frac{\partial w_0}{\partial \xi_\beta}\frac{\partial^2 w_1}{\partial x_\alpha \partial x_\beta}\Big) \notag\\
        &\quad + \frac{\partial b}{\partial x_\alpha} w_0 + b \frac{\partial w_0}{\partial x_\alpha}+ \frac{\partial^2 w_0}{\partial x_n\partial x_\alpha} - \frac{\partial c_0}{\partial x_\alpha}
    \bigg],\\
    \label{a12}\frac{\partial w_{-1}}{\partial x_n}(\phi,V)
    &=-\frac{1}{2}w_1^{-2}\frac{\partial w_1}{\partial x_n}
        \bigg(
            -w_0^2 + i\sum_\alpha \frac{\partial w_1}{\partial \xi_\alpha}\frac{\partial w_0}{\partial x_\alpha}  + bw_0 + \frac{\partial w_0}{\partial x_n} - c_0 
        \bigg) \notag\\
        &\quad +\frac{1}{2} w_1^{-1} 
        \bigg[
            - 2w_0\frac{\partial w_0}{\partial x_n} + i\sum_\alpha \Big(\frac{\partial^2 w_1}{\partial \xi_\alpha\partial x_n}\frac{\partial w_0}{\partial x_\alpha} + \frac{\partial w_1}{\partial \xi_\alpha}\frac{\partial^2 w_0}{\partial x_n\partial x_\alpha}\Big) \notag\\
            & \quad + \frac{\partial b}{\partial x_n} w_0 + b\frac{\partial w_0}{\partial x_n} + \frac{\partial^2 w_0}{\partial x_n^2} - \frac{\partial c_0}{\partial x_n}
        \bigg].
\end{align}
Substituting \eqref{a8}, \eqref{a10}, and \eqref{a12} into \eqref{a9}, we have
\begin{align}\label{a17}
    &s_{-1}w_{-2}(\phi,V)
    \cong \frac{1}{2}s_{-1}w_1^{-2}w_0^3 -is_{-1}w_1^{-2}w_0\sum_{\alpha}\frac{\partial w_1}{\partial \xi_\alpha}\frac{\partial w_{0}}{\partial x_\alpha}\notag\\
    & \quad -\frac{1}{4}s_{-1}w_1^{-2}w_0\sum_{\alpha,\beta} \frac{\partial^2 w_1}{\partial \xi_{\alpha} \partial \xi_{\beta}} \frac{\partial^2 w_1}{\partial x_{\alpha} \partial x_{\beta}} -\frac{3}{4}s_{-1}w_1^{-2}bw_0^2 -s_{-1}w_1^{-2}w_0\frac{\partial w_0}{\partial x_n} +\frac{1}{2}s_{-1}w_1^{-2} w_0c_0 \notag\\
    & \quad -\frac{1}{4}s_{-1}w_1^{-2}\sum_{\alpha,\beta}\frac{\partial w_1}{\partial \xi_\alpha}\frac{\partial w_1}{\partial \xi_\beta}\frac{\partial^2 w_0}{\partial x_{\alpha} \partial x_{\beta}}+\frac{i}{4}s_{-1}w_1^{-2}w_0\sum_{\alpha}\frac{\partial w_1}{\partial \xi_\alpha}\frac{\partial b}{\partial x_\alpha}\notag\\
    & \quad +\frac{i}{2}s_{-1}w_1^{-2}b\sum_{\alpha}\frac{\partial w_1}{\partial \xi_\alpha}\frac{\partial w_0}{\partial x_\alpha}+\frac{i}{2}s_{-1}w_1^{-2}\sum_{\alpha}\frac{\partial w_1}{\partial \xi_\alpha}\frac{\partial^2 w_0}{\partial x_{n} \partial x_{\alpha}}+\frac{i}{2}s_{-1}w_1^{-1}\sum_{\alpha}\frac{\partial w_0}{\partial \xi_\alpha}\frac{\partial w_0}{\partial x_{\alpha}}\notag\\
    & \quad +\frac{1}{4}s_{-1}w_1^{-1}\sum_{\alpha,\beta}\frac{\partial^2 w_1}{\partial \xi_\alpha\partial \xi_\beta}\frac{\partial^2 w_0}{\partial x_{\alpha}\partial x_{\beta}}+\frac{1}{8}s_{-1}w_1^{-2}b\sum_{\alpha,\beta}\frac{\partial^2 w_1}{\partial \xi_\alpha\partial \xi_\beta}\frac{\partial^2 w_1}{\partial x_{\alpha}\partial x_{\beta}}+\frac{1}{4}s_{-1}w_1^{-2}b^2w_0\notag\\
    & \quad +\frac{1}{2}s_{-1}w_1^{-2}b\frac{\partial w_0}{\partial x_n}-\frac{1}{4}s_{-1}w_1^{-2}bc_0+\frac{1}{4}s_{-1}w_1^{-3}\frac{\partial w_1}{\partial x_n}w_0^2-\frac{i}{4}s_{-1}w_1^{-3}\frac{\partial w_1}{\partial x_n}\sum_{\alpha}\frac{\partial w_1}{\partial \xi_\alpha}\frac{\partial w_0}{\partial x_{\alpha}}\notag\\
    & \quad -\frac{1}{4}s_{-1}w_1^{-3}\frac{\partial w_1}{\partial x_n}bw_0-\frac{1}{4}s_{-1}w_1^{-3}\frac{\partial w_1}{\partial x_n}\frac{\partial w_0}{\partial x_n}+\frac{1}{4}s_{-1}w_1^{-3}\frac{\partial w_1}{\partial x_n}c_0\notag\\
    & \quad +\frac{i}{4}s_{-1}w_1^{-2}\sum_{\alpha}\frac{\partial^2 w_1}{\partial \xi_\alpha\partial x_n}\frac{\partial w_0}{\partial x_{\alpha}}+\frac{1}{4}s_{-1}w_1^{-2}\Big(\frac{\partial b}{\partial x_n}w_0+\frac{\partial^2 w_0}{\partial x_n^2}-\frac{\partial c_0}{\partial x_n}\Big).
\end{align}

It follows from \eqref{2.17} that
\begin{align}\label{a13}
    \frac{\partial s_{-3}}{\partial x_\alpha}(\phi,V)
    &= -s_{-1}
    \bigg[
        \frac{\partial w_0}{\partial x_\alpha}s_{-2} + w_0\frac{\partial s_{-2}}{\partial x_\alpha} + \frac{\partial w_{-1}}{\partial x_\alpha}s_{-1} - i \sum_\beta
        \bigg(
            \frac{\partial w_1}{\partial \xi_\beta} \frac{\partial^2 s_{-2}}{\partial x_\alpha\partial x_\beta} \notag\\
            &\qquad + \frac{\partial w_0}{\partial \xi_\beta} \frac{\partial^2 s_{-1}}{\partial x_\alpha\partial x_\beta}
        \bigg)
    \bigg]. 
\end{align}
From \eqref{2.16} we get
\begin{align}\label{8.39}
    \frac{\partial^2 s_{-2}}{\partial x_\alpha\partial x_\beta}(\phi,V)
    =2s_{-1}^3w_0\frac{\partial^2 w_1}{\partial x_\alpha\partial x_\beta} -s_{-1}^2\frac{\partial^2 w_0}{\partial x_\alpha\partial x_\beta}.
\end{align}

Substituting \eqref{8.4}, \eqref{a5}, \eqref{a6}, \eqref{a10}, and \eqref{8.39} into \eqref{a13} we obtain
\begin{align}\label{a14}
    \frac{\partial s_{-3}}{\partial x_\alpha}(\phi,V)
    &= (2s_{-1}^3+s_{-1}^2w_1^{-1})w_0\frac{\partial w_0}{\partial x_{\alpha}} +3is_{-1}^4w_0\sum_\beta \frac{\partial w_1}{\partial \xi_\beta}\frac{\partial^2 w_1}{\partial x_{\alpha} \partial x_{\beta}} \notag\\
    &\quad-i\Big(s_{-1}^3+\frac{1}{2}s_{-1}^2w_1^{-1}\Big)\sum_{\beta}\Big(\frac{\partial w_1}{\partial \xi_\beta}\frac{\partial^2 w_0}{\partial x_{\alpha} \partial x_{\beta}}+\frac{\partial w_0}{\partial \xi_\beta}\frac{\partial^2 w_1}{\partial x_{\alpha} \partial x_{\beta}}\Big)\notag\\
    &\quad -\frac{1}{2}s_{-1}^2w_1^{-1}\Big(\frac{\partial b}{\partial x_{\alpha}}w_0+b\frac{\partial w_0}{\partial x_{\alpha}}+\frac{\partial^2 w_0}{\partial x_{n}\partial x_{\alpha}}-\frac{\partial c_0}{\partial x_{\alpha}}\Big).
\end{align}

Since $\sum_{\alpha}\frac{\partial w_1}{\partial \xi_{\alpha}}\frac{\partial c_0}{\partial x_{\alpha}}$ is odd in $\xi^{\prime}$, then combining \eqref{16.6}, \eqref{16.7}, and \eqref{a14} we have
\begin{align}\label{a18}
    &\Big(-i\sum_{\alpha}\frac{\partial w_1}{\partial \xi_{\alpha}}\frac{\partial s_{-3}}{\partial x_\alpha}\Big)(\phi,V)
    \cong -i(2s_{-1}^3+s_{-1}^2w_1^{-1})w_0\sum_{\alpha}\frac{\partial w_1}{\partial \xi_{\alpha}}\frac{\partial w_0}{\partial x_{\alpha}} \notag\\
    &\quad-\Big(s_{-1}^3+\frac{1}{2}s_{-1}^2w_1^{-1}\Big)\sum_{\alpha,\beta}\frac{\partial w_1}{\partial \xi_{\alpha}}\frac{\partial w_1}{\partial \xi_\beta}\frac{\partial^2 w_0}{\partial x_{\alpha} \partial x_{\beta}} +\frac{i}{2}s_{-1}^2w_1^{-1}\Big(w_0\sum_{\alpha}\frac{\partial w_1}{\partial \xi_{\alpha}}\frac{\partial b}{\partial x_{\alpha}} \notag\\
    &\quad +b\sum_{\alpha}\frac{\partial w_1}{\partial \xi_{\alpha}}\frac{\partial w_0}{\partial x_{\alpha}}+\sum_{\alpha}\frac{\partial w_1}{\partial \xi_{\alpha}}\frac{\partial^2 w_0}{\partial x_{n}\partial x_{\alpha}}\Big).
\end{align}

Combining \eqref{a5} and \eqref{16.7} we have
\begin{align}\label{a19}
    \Big(-i\sum_{\alpha}\frac{\partial w_0}{\partial \xi_{\alpha}}\frac{\partial s_{-2}}{\partial x_\alpha}\Big)(\phi,V)
    \cong is_{-1}^2\sum_{\alpha}\frac{\partial w_0}{\partial \xi_{\alpha}}\frac{\partial w_0}{\partial x_\alpha}.
\end{align}
By \eqref{8.39} we get 
\begin{align}\label{a20}
    \Big(-\frac{1}{2}\sum_{\alpha,\beta}\frac{\partial^2 w_1}{\partial \xi_\alpha\partial \xi_\beta}\frac{\partial^2 s_{-2}}{\partial x_{\alpha} \partial x_{\beta}}\Big)(\phi,V)
    &\cong -s_{-1}^3w_0\sum_{\alpha,\beta}\frac{\partial^2 w_1}{\partial \xi_\alpha\partial \xi_\beta}\frac{\partial^2 w_1}{\partial x_{\alpha} \partial x_{\beta}} \notag\\
    &\quad +\frac{1}{2}s_{-1}^2\sum_{\alpha,\beta}\frac{\partial^2 w_1}{\partial \xi_\alpha\partial \xi_\beta}\frac{\partial^2 w_0}{\partial x_{\alpha} \partial x_{\beta}}.
\end{align}
Substituting \eqref{a15}, \eqref{a16}, \eqref{a17}, \eqref{a18}, \eqref{a19}, and \eqref{a20} into \eqref{8.32}, we then have
\begin{align}\label{a25}
    &s_{-4}(\phi,V)
    \cong -\Big(s_{-1}^4+s_{-1}^3w_1^{-1}+\frac{1}{2}s_{-1}^2w_1^{-2}\Big)w_0^3 +i\Big(3s_{-1}^4+2s_{-1}^3w_1^{-1}+s_{-1}^2w_1^{-2}\Big) \notag\\
    &\quad \times w_0\sum_\alpha\frac{\partial w_1}{\partial \xi_\alpha}\frac{\partial w_0}{\partial x_\alpha} +\Big(\frac{3}{2}s_{-1}^4+\frac{1}{2}s_{-1}^3w_1^{-1}+\frac{1}{4}s_{-1}^2w_1^{-2}\Big)w_0\sum_{\alpha,\beta}\frac{\partial^2 w_1}{\partial \xi_\alpha\partial \xi_\beta}\frac{\partial^2 w_1}{\partial x_{\alpha} \partial x_{\beta}}  \notag\\
    &\quad +\Big(s_{-1}^3w_1^{-1}+\frac{3}{4}s_{-1}^2w_1^{-2}\Big)bw_0^2 +\Big(s_{-1}^3w_1^{-1}+s_{-1}^2w_1^{-2}\Big)w_0 \frac{\partial w_0}{\partial x_{n}} -\Big(s_{-1}^3w_1^{-1}+\frac{1}{2}s_{-1}^2w_1^{-2}\Big) \notag\\
    &\quad \times w_0c_0 +\Big(s_{-1}^4+\frac{1}{2}s_{-1}^3w_1^{-1}+\frac{1}{4}s_{-1}^2w_1^{-2}\Big)\sum_{\alpha,\beta}\frac{\partial w_1}{\partial \xi_\alpha}\frac{\partial w_1}{\partial \xi_\beta}\frac{\partial^2 w_0}{\partial x_\alpha\partial x_\beta}  \notag\\
    &\quad -i\Big(\frac{1}{2}s_{-1}^3w_1^{-1}+\frac{1}{4}s_{-1}^2w_1^{-2}\Big)w_0\sum_\alpha\frac{\partial w_1}{\partial \xi_\alpha}\frac{\partial b}{\partial x_\alpha} -\frac{i}{2}\Big(s_{-1}^3w_1^{-1}+s_{-1}^2w_1^{-2}\Big)b\sum_\alpha\frac{\partial w_1}{\partial \xi_\alpha}\frac{\partial w_0}{\partial x_\alpha}  \notag\\
    &\quad -\frac{i}{2}\Big(s_{-1}^3w_1^{-1}+s_{-1}^2w_1^{-2}\Big)\sum_\alpha\frac{\partial w_1}{\partial \xi_\alpha}\frac{\partial^2 w_0}{\partial x_n \partial x_\alpha} -i\Big(s_{-1}^3+\frac{1}{2}s_{-1}^2w_1^{-1}\Big)\sum_\alpha\frac{\partial w_0}{\partial \xi_\alpha}\frac{\partial w_0}{\partial x_\alpha}  \notag\\
    &\quad -\Big(\frac{1}{2}s_{-1}^3+\frac{1}{4}s_{-1}^2w_1^{-1}\Big)\sum_{\alpha,\beta}\frac{\partial^2 w_1}{\partial \xi_\alpha\partial \xi_\beta}\frac{\partial^2 w_0}{\partial x_{\alpha} \partial x_{\beta}} -\frac{1}{8}s_{-1}^2w_1^{-2}b\sum_{\alpha,\beta}\frac{\partial^2 w_1}{\partial \xi_\alpha\partial \xi_\beta}\frac{\partial^2 w_1}{\partial x_{\alpha} \partial x_{\beta}}  \notag\\
    &\quad -\frac{1}{4}s_{-1}^2w_1^{-2}b^2w_0 -\frac{1}{2}s_{-1}^2w_1^{-2}b\frac{\partial w_0}{\partial x_n} +\frac{1}{4}s_{-1}^2w_1^{-2}bc_0 -\frac{1}{4}s_{-1}^2w_1^{-3}\frac{\partial w_1}{\partial x_n}w_0^2  \notag\\
    &\quad +\frac{i}{4}s_{-1}^2w_1^{-3}\frac{\partial w_1}{\partial x_n}\sum_\alpha\frac{\partial w_1}{\partial \xi_\alpha}\frac{\partial w_0}{\partial x_\alpha} +\frac{1}{4}s_{-1}^2w_1^{-3}\frac{\partial w_1}{\partial x_n}bw_0 +\frac{1}{4}s_{-1}^2w_1^{-3}\frac{\partial w_1}{\partial x_n}\frac{\partial w_0}{\partial x_n}  \notag\\
    &\quad -\frac{1}{4}s_{-1}^2w_1^{-3}\frac{\partial w_1}{\partial x_n}c_0 -\frac{i}{4}s_{-1}^2w_1^{-2}\sum_\alpha\frac{\partial^2 w_1}{\partial \xi_\alpha \partial x_n}\frac{\partial w_0}{\partial x_\alpha}  \notag\\
    &\quad -\frac{1}{4}s_{-1}^2w_1^{-2}\Big(\frac{\partial b}{\partial x_n}w_0+\frac{\partial^2 w_0}{\partial x_n^2}-\frac{\partial c_0}{\partial x_n}\Big).
\end{align}

By \eqref{b5} we get 
\begin{align}
    w_0^3(\phi,V)
    \cong \frac{1}{8}\Big[b^3+3b^2w_1^{-1}\frac{\partial w_1}{\partial x_n}+3bw_1^{-2}\Big(\frac{\partial w_1}{\partial x_n}\Big)^2+3bw_1^{-2}c_1^2+3w_1^{-3}\frac{\partial w_1}{\partial x_n}c_1^2\Big].
\end{align}
From \eqref{2.7}, \eqref{8.7}, and \eqref{7.1.3} we compute
\begin{align}\label{a21}
    \frac{\partial w_0}{\partial x_\alpha}(x_0)
        = \frac{1}{2} w_1^{-1} 
        \biggl(
            i\sum_{\beta}\frac{\partial w_1}{\partial \xi_\beta}\frac{\partial^2 w_1}{\partial x_\alpha \partial x_\beta} + \frac{\partial b}{\partial x_\alpha}w_1  - \frac{\partial c_1}{\partial x_\alpha}
        \biggr).
\end{align}
Combining \eqref{a21} and \eqref{16.6} we then get
\begin{align}\label{a22}
    \Big(\sum_\alpha\frac{\partial w_1}{\partial \xi_\alpha}\frac{\partial w_0}{\partial x_\alpha}\Big)(x_0)
    \cong \frac{1}{2}\sum_\alpha\frac{\partial w_1}{\partial \xi_\alpha}\frac{\partial b}{\partial x_\alpha} -\frac{1}{2}w_1^{-1}\sum_\alpha\frac{\partial w_1}{\partial \xi_\alpha}\frac{\partial c_1}{\partial x_\alpha}.
\end{align}
Thus, it follows from \eqref{a22} and \eqref{b5} we have
\begin{align}
    &\Big(w_0\sum_\alpha\frac{\partial w_1}{\partial \xi_\alpha}\frac{\partial w_0}{\partial x_\alpha}\Big)(\phi,V)
    \cong -\frac{1}{4}w_1^{-1}c_1\sum_\alpha\frac{\partial w_1}{\partial \xi_\alpha}\frac{\partial b}{\partial x_\alpha} -\frac{1}{4}w_1^{-1}b\sum_\alpha\frac{\partial w_1}{\partial \xi_\alpha}\frac{\partial c_1}{\partial x_\alpha} \notag\\
    &\qquad -\frac{1}{4}w_1^{-2}\frac{\partial w_1}{\partial x_n}\sum_\alpha\frac{\partial w_1}{\partial \xi_\alpha}\frac{\partial c_1}{\partial x_\alpha},\\
    &\Big(w_0\sum_{\alpha,\beta}\frac{\partial^2 w_1}{\partial \xi_\alpha\partial \xi_\beta}\frac{\partial^2 w_1}{\partial x_{\alpha} \partial x_{\beta}}\Big)(\phi,V)
    \cong \frac{1}{2}b\sum_{\alpha,\beta}\frac{\partial^2 w_1}{\partial \xi_\alpha\partial \xi_\beta}\frac{\partial^2 w_1}{\partial x_{\alpha} \partial x_{\beta}},\\
    &(bw_0^2)(\phi,V)
    \cong \frac{1}{4}\Big[b^3+bw_1^{-2}\Big(\frac{\partial w_1}{\partial x_n}\Big)^2+bw_1^{-2}c_1^2+2b^2w_1^{-1}\frac{\partial w_1}{\partial x_n}\Big].
\end{align}
By \eqref{2.7} we compute
\begin{align}\label{a24}
    \frac{\partial w_0}{\partial x_n}(x_0)
    = -\frac{1}{2}w_1^{-2}\Big(\frac{\partial w_1}{\partial x_n}\Big)^2 +\frac{1}{2}w_1^{-2}\frac{\partial w_1}{\partial x_n}c_1 +\frac{1}{2}\frac{\partial b}{\partial x_n} +\frac{1}{2}w_1^{-1}\frac{\partial^2 w_1}{\partial x_n^2} -\frac{1}{2}w_1^{-1}\frac{\partial c_1}{\partial x_n}.
\end{align}
Then,
\begin{align}
    \Big(w_0\frac{\partial w_0}{\partial x_n}\Big)(\phi,V)
    &\cong -\frac{1}{4}w_1^{-2}\Big(\frac{\partial w_1}{\partial x_n}\Big)^2b +\frac{1}{4}b\frac{\partial b}{\partial x_n} +\frac{1}{4}w_1^{-1}b\frac{\partial^2 w_1}{\partial x_n^2} +\frac{1}{4}w_1^{-1}\frac{\partial b}{\partial x_n}\frac{\partial w_1}{\partial x_n} \notag\\
    &\quad -\frac{1}{4}w_1^{-3}\frac{\partial w_1}{\partial x_n}c_1^2 +\frac{1}{4}c_1\frac{\partial c_1}{\partial x_n}.
\end{align}
From \eqref{b5} and \eqref{6.9} we have
\begin{align}
    (w_0c_0)(\phi,V)
    \cong \frac{1}{2}\Big(b+w_1^{-1}\frac{\partial w_1}{\partial x_n}\Big)V.
\end{align}
It follows from \eqref{2.7} that 
\begin{align}\label{a23}
    \frac{\partial^2 w_0}{\partial x_{\alpha}\partial x_{\beta}}(\phi,V)
    \cong \frac{1}{2} \frac{\partial^2 b}{\partial x_{\alpha}\partial x_{\beta}}.
\end{align}
We then have
\begin{align}
    \Big(\sum_{\alpha,\beta}\frac{\partial w_1}{\partial \xi_{\alpha}}\frac{\partial w_1}{\partial \xi_{\beta}}\frac{\partial^2 w_0}{\partial x_{\alpha}\partial x_{\beta}}\Big)(\phi,V)
    \cong \frac{1}{2}\sum_{\alpha,\beta}\frac{\partial w_1}{\partial \xi_{\alpha}}\frac{\partial w_1}{\partial \xi_{\beta}}\frac{\partial^2 b}{\partial x_{\alpha}\partial x_{\beta}}.
\end{align}
By \eqref{b5} we get
\begin{align}
    \Big(w_0\sum_{\alpha}\frac{\partial w_1}{\partial \xi_{\alpha}}\frac{\partial b}{\partial x_{\alpha}}\Big)(\phi,V)
    \cong -\frac{1}{2}w_1^{-1}c_1\sum_{\alpha}\frac{\partial w_1}{\partial \xi_{\alpha}}\frac{\partial b}{\partial x_{\alpha}}.
\end{align}
From \eqref{a21} we have
\begin{align}
    \Big(b\sum_{\alpha}\frac{\partial w_1}{\partial \xi_{\alpha}}\frac{\partial w_0}{\partial x_{\alpha}}\Big)(\phi,V)
    \cong -\frac{1}{2}w_1^{-1}b\sum_{\alpha}\frac{\partial w_1}{\partial \xi_{\alpha}}\frac{\partial c_1}{\partial x_{\alpha}}.
\end{align}
From \eqref{2.7} we see that 
\begin{align}
    \frac{\partial^2 w_0}{\partial x_{n} \partial x_{\alpha}}(\phi,V)
    =\frac{1}{2}\frac{\partial^2 }{\partial x_{n} \partial x_{\alpha}}(b-w_1^{-1}c_1).
\end{align}
Thus,
\begin{align}
    \Big(\sum_{\alpha}\frac{\partial w_1}{\partial \xi_{\alpha}}\frac{\partial^2 w_0}{\partial x_{n} \partial x_{\alpha}}\Big)(\phi,V)
    &\cong -\frac{1}{2}w_1^{-1}\sum_{\alpha}\frac{\partial w_1}{\partial \xi_{\alpha}}\frac{\partial^2 c_1}{\partial x_{n} \partial x_{\alpha}}  +\frac{1}{2}w_1^{-2}\frac{\partial w_1}{\partial x_{n}}\sum_{\alpha}\frac{\partial w_1}{\partial \xi_{\alpha}}\frac{\partial c_1}{\partial x_{\alpha}}.
\end{align}

It follows from \eqref{2.7} and \eqref{a21} that
\begin{align}
    \Big(\sum_{\alpha}\frac{\partial w_0}{\partial \xi_{\alpha}}\frac{\partial w_0}{\partial x_{\alpha}}\Big)(\phi,V)
    &\cong \frac{1}{4}w_1^{-2}c_1\sum_{\alpha}\frac{\partial w_1}{\partial \xi_{\alpha}}\frac{\partial b}{\partial x_{\alpha}} -\frac{1}{4}w_1^{-1}\sum_{\alpha}\frac{\partial c_1}{\partial \xi_{\alpha}}\frac{\partial b}{\partial x_{\alpha}} \notag\\
    &\quad +\frac{1}{4}w_1^{-3}\frac{\partial w_1}{\partial x_{n}}\sum_{\alpha}\frac{\partial w_1}{\partial \xi_{\alpha}}\frac{\partial c_1}{\partial x_{\alpha}} -\frac{1}{4}w_1^{-2}\sum_{\alpha}\frac{\partial^2 w_1}{\partial \xi_{\alpha}\partial x_n}\frac{\partial c_1}{\partial x_{\alpha}}.
\end{align}
By \eqref{a23} we have
\begin{align}
    \Big(\sum_{\alpha,\beta}\frac{\partial^2 w_1}{\partial \xi_{\alpha} \partial \xi_{\beta}}\frac{\partial^2 w_0}{\partial x_{\alpha}\partial x_{\beta}}\Big)(\phi,V)
    \cong \frac{1}{2}\sum_{\alpha,\beta}\frac{\partial^2 w_1}{\partial \xi_{\alpha} \partial \xi_{\beta}}\frac{\partial^2 b}{\partial x_{\alpha}\partial x_{\beta}}.
\end{align}
By \eqref{b5} we get
\begin{align}
    (b^2w_0)(\phi,V)
    \cong \frac{1}{2}b^3+\frac{1}{2}w_1^{-1}b^2\frac{\partial w_1}{\partial x_{n}}.
\end{align}
Recall that \eqref{a24}, we then obtain
\begin{align}
    \Big(b\frac{\partial w_0}{\partial x_{n}}\Big)(\phi,V)
    \cong -\frac{1}{2}w_1^{-2}b\Big(\frac{\partial w_1}{\partial x_{n}}\Big)^2 +\frac{1}{2}b\frac{\partial b}{\partial x_{n}} +\frac{1}{2}w_1^{-1}b\frac{\partial^2 w_1}{\partial x_{n}^2}.
\end{align}
From \eqref{b6} we have
\begin{align}
    \Big(\frac{\partial w_1}{\partial x_{n}}w_0^2\Big)(\phi,V)
    \cong \frac{1}{4}\frac{\partial w_1}{\partial x_{n}}b^2 +\frac{1}{4}w_1^{-2}c_1^2\frac{\partial w_1}{\partial x_{n}} +\frac{1}{2}w_1^{-1}b\Big(\frac{\partial w_1}{\partial x_{n}}\Big)^2.
\end{align}
Using \eqref{a22} we get 
\begin{align}
    \Big(\frac{\partial w_1}{\partial x_{n}}\sum_{\alpha}\frac{\partial w_1}{\partial \xi_{\alpha}}\frac{\partial w_0}{\partial x_{\alpha}}\Big)(\phi,V)
    \cong -\frac{1}{2}w_1^{-1}\frac{\partial w_1}{\partial x_{n}}\sum_{\alpha}\frac{\partial w_1}{\partial \xi_{\alpha}}\frac{\partial c_1}{\partial x_{\alpha}}.
\end{align}
By \eqref{b5} we obtain
\begin{align}
    \Big(\frac{\partial w_1}{\partial x_n}bw_0\Big)(\phi,V)
    \cong \frac{1}{2}b^2\frac{\partial w_1}{\partial x_n} +\frac{1}{2}w_1^{-1}b\Big(\frac{\partial w_1}{\partial x_n}\Big)^2.
\end{align}
From \eqref{a24} we have
\begin{align}
    \Big(\frac{\partial w_1}{\partial x_n}\frac{\partial w_0}{\partial x_n}\Big)(\phi,V)
    \cong \frac{1}{2}\frac{\partial w_1}{\partial x_n}\frac{\partial b}{\partial x_n}.
\end{align}
Using \eqref{a21} we get 
\begin{align}
    \Big(\sum_{\alpha}\frac{\partial^2 w_1}{\partial \xi_{\alpha}\partial x_n}\frac{\partial w_0}{\partial x_{\alpha}}\Big)(\phi,V)
    \cong -\frac{1}{2}w_1^{-1}\sum_{\alpha}\frac{\partial^2 w_1}{\partial \xi_{\alpha}\partial x_n}\frac{\partial c_1}{\partial x_{\alpha}}.
\end{align}
By \eqref{b5} we see that 
\begin{align}
    \Big(\frac{\partial b}{\partial x_n}w_0\Big)(\phi,V)
    \cong \frac{1}{2}b\frac{\partial b}{\partial x_n} +\frac{1}{2}w_1^{-1}\frac{\partial b}{\partial x_n}\frac{\partial w_1}{\partial x_n}.
\end{align}
From \eqref{2.7} we get
\begin{align}
    \frac{\partial^2 w_0}{\partial x_n^2}(\phi,V)
    \cong \frac{1}{2}\frac{\partial^2 b}{\partial x_n^2}.
\end{align}
Substituting the above results into \eqref{a25} we have 
\begin{align}\label{a27}
    &s_{-4}(\phi,V)
    \cong -\frac{1}{8}(s_{-1}^4 -s_{-1}^3w_1^{-1})b^3 -\frac{1}{8}(3s_{-1}^4w_1^{-1} -s_{-1}^3w_1^{-2} -s_{-1}^2w_1^{-3})b^2\frac{\partial w_1}{\partial x_n} \notag\\
    &\quad -\frac{3}{8}(s_{-1}^4w_1^{-2} +s_{-1}^3w_1^{-3})b\Big(\frac{\partial w_1}{\partial x_n}\Big)^2 -\frac{1}{8}(3s_{-1}^4w_1^{-2} +s_{-1}^3w_1^{-3})bc_1^2 -\frac{1}{8}(3s_{-1}^4w_1^{-3}  \notag\\
    &\quad +5s_{-1}^3w_1^{-4} +4s_{-1}^2w_1^{-5})c_1^2\frac{\partial w_1}{\partial x_n} -\frac{i}{4}(3s_{-1}^4w_1^{-1} +2s_{-1}^3w_1^{-2} +s_{-1}^2w_1^{-3})c_1\sum_{\alpha}\frac{\partial w_1}{\partial \xi_{\alpha}}\frac{\partial b}{\partial x_{\alpha}}  \notag\\
    &\quad -\frac{i}{4}(3s_{-1}^4w_1^{-1} +s_{-1}^3w_1^{-2})b\sum_{\alpha}\frac{\partial w_1}{\partial \xi_{\alpha}}\frac{\partial c_1}{\partial x_{\alpha}} -\frac{i}{4}(3s_{-1}^4w_1^{-2} +s_{-1}^3w_1^{-3} +3s_{-1}^2w_1^{-4}) \notag\\
    &\quad \times \frac{\partial w_1}{\partial x_n}\sum_{\alpha}\frac{\partial w_1}{\partial \xi_{\alpha}}\frac{\partial c_1}{\partial x_{\alpha}} +\frac{1}{8}(2s_{-1}^3w_1^{-1} -s_{-1}^2w_1^{-2})b\frac{\partial b}{\partial x_n} +\frac{1}{4}s_{-1}^3w_1^{-2}b\frac{\partial^2 w_1}{\partial x_n^2}  \notag\\
    &\quad +\frac{1}{4}(s_{-1}^3w_1^{-2}+s_{-1}^2w_1^{-3})\frac{\partial w_1}{\partial x_n}\frac{\partial b}{\partial x_n} +\frac{1}{4}(s_{-1}^3w_1^{-3}+s_{-1}^2w_1^{-4})c_1\frac{\partial c_1}{\partial x_n} -\frac{1}{2}s_{-1}^3w_1^{-1}bV  \notag\\
    &\quad -\frac{1}{2}(s_{-1}^3w_1^{-2}+s_{-1}^2w_1^{-3})V\frac{\partial w_1}{\partial x_n} +\frac{1}{8}(4s_{-1}^4 +2s_{-1}^3w_1^{-1} +s_{-1}^2w_1^{-2})\sum_{\alpha,\beta}\frac{\partial w_1}{\partial \xi_\alpha}\frac{\partial w_1}{\partial \xi_\beta}\frac{\partial^2 b}{\partial x_{\alpha} \partial x_{\beta}}  \notag\\
    &\quad +\frac{i}{4}(s_{-1}^3w_1^{-2} +s_{-1}^2w_1^{-3})\sum_{\alpha}\frac{\partial w_1}{\partial \xi_\alpha}\frac{\partial^2 c_1}{\partial x_{n} \partial x_{\alpha}} +\frac{i}{4}(s_{-1}^3w_1^{-1} +s_{-1}^2w_1^{-2})\sum_{\alpha}\frac{\partial c_1}{\partial \xi_\alpha}\frac{\partial b}{\partial x_{\alpha}}   \notag\\
    &\quad +\frac{i}{4}(s_{-1}^3w_1^{-2} +s_{-1}^2w_1^{-3}) \sum_{\alpha}\frac{\partial^2 w_1}{\partial \xi_\alpha\partial x_n}\frac{\partial c_1}{\partial x_{\alpha}} -\frac{1}{8}(2s_{-1}^3+s_{-1}^2w_1^{-1})\sum_{\alpha,\beta}\frac{\partial^2 w_1}{\partial \xi_\alpha\partial \xi_\beta}\frac{\partial^2 b}{\partial x_{\alpha} \partial x_{\beta}}  \notag\\
    &\quad -\frac{1}{8}s_{-1}^2w_1^{-2}\frac{\partial^2 b}{\partial x_n^2} +\frac{1}{4}s_{-1}^2w_1^{-2}\frac{\partial V}{\partial x_n} +\frac{1}{4}(3s_{-1}^4 +s_{-1}^3w_1^{-1})b\sum_{\alpha,\beta}\frac{\partial^2 w_1}{\partial \xi_\alpha\partial \xi_\beta}\frac{\partial^2 w_1}{\partial x_{\alpha} \partial x_{\beta}}.
\end{align}

By \eqref{3.12.1} we have
\begin{align}
    b^3(\phi,V)
    \cong -(3H^2\phi_n+3H\phi_n^2+\phi_n^3).
\end{align}
Combining \eqref{8.7} and \eqref{3.12.1} we get
\begin{align}
    &\Big(b^2\frac{\partial w_1}{\partial x_n}\Big)(\phi,V)
    \cong (2H\phi_n+\phi_n^2)w_1^{-1}\sum_{\alpha}\kappa_{\alpha}\xi_{\alpha}^2,\\
    &b\Big(\frac{\partial w_1}{\partial x_n}\Big)^2(\phi,V)
    \cong -\phi_n w_1^{-2}\sum_{\alpha,\beta}\kappa_{\alpha}\kappa_\beta\xi_{\alpha}^2\xi_{\beta}^2.
\end{align}
From \eqref{3.12.1} and \eqref{3.12.2} we have
\begin{align}
    (bc_1^2)(\phi,V)
    \cong (H+\phi_n)\sum_{\alpha}\phi_{\alpha}^2\xi_{\alpha}^2.
\end{align}
Combining \eqref{8.7} and \eqref{3.12.2} we get
\begin{align}
    \Big(\frac{\partial w_1}{\partial x_n}c_1^2\Big)(\phi,V)
    \cong -w_1^{-1}\sum_{\alpha,\beta}\kappa_{\alpha}\phi_\beta^2\xi_{\alpha}^2\xi_{\beta}^2.
\end{align}
By \eqref{3.12.2}, \eqref{8.15}, and \eqref{6.6}, we have
\begin{align}
    \Big(ic_1\sum_{\alpha}\frac{\partial w_1}{\partial \xi_{\alpha}}\frac{\partial b}{\partial x_{\alpha}}\Big)(\phi,V)
    \cong -w_1^{-1}\sum_{\alpha}\phi_{\alpha}\phi_{n\alpha}\xi_{\alpha}^2.
\end{align}

It follows from \eqref{6.8} that 
\begin{align*}
    \frac{\partial c_1}{\partial x_{\alpha}}(x_0)
    = i\sum_{\beta,\gamma}\Big(\frac{1}{2}g_{\gamma\gamma,\alpha\beta} + g^{\gamma\beta,\gamma\alpha}\Big)\xi_\beta - i \sum_{\beta}\phi_{\alpha\beta} \xi_\beta.
\end{align*}
Using \eqref{7.1.5} and \eqref{7.1.1.1}, we have 
\begin{align}\label{a26}
    \frac{\partial c_1}{\partial x_{\alpha}}(x_0)
    =-i\sum_\beta \Big(\frac{2}{3}R_{\alpha\beta}+\phi_{\alpha\beta}\Big)\xi_\beta.
\end{align}
Hence, combining this and \eqref{3.12.1}, \eqref{8.15}, and \eqref{8.7} we obtain
\begin{align}
    &\Big(ib\sum_{\alpha}\frac{\partial w_1}{\partial \xi_{\alpha}}\frac{\partial c_1}{\partial x_{\alpha}}\Big)(\phi,V)
    \cong -w_{1}^{-1}H\sum_{\alpha}\phi_{\alpha\alpha}\xi_{\alpha}^2 -w_{1}^{-1}\phi_n\sum_{\alpha}\Big(\frac{2}{3}R_{\alpha\alpha}+\phi_{\alpha\alpha}\Big)\xi_{\alpha}^2,\\
    &\Big(i\frac{\partial w_1}{\partial x_n}\sum_{\alpha}\frac{\partial w_1}{\partial \xi_{\alpha}}\frac{\partial c_1}{\partial x_{\alpha}}\Big)(\phi,V)
    \cong w_{1}^{-2}\sum_{\alpha}\kappa_\alpha\phi_{\beta\beta}\xi_{\alpha}^2\xi_{\beta}^2.
\end{align}
Using \eqref{6.6}, \eqref{6.3.2}, and \eqref{3.12.1} we get
\begin{align}
    \Big(b\frac{\partial b}{\partial x_n}\Big)(\phi,V)
    \cong H\phi_{nn}+\phi_n\Big(2\sum_{\alpha}\kappa_\alpha^2-\frac{1}{2}\sum_{\alpha}g_{\alpha\alpha,nn}+\phi_{nn}\Big).
\end{align}
From \eqref{3.12.1} and \eqref{7.1.4} we have
\begin{align}
    \Big(b\frac{\partial^2 w_1}{\partial x_n^2}\Big)(\phi,V)
    \cong -\frac{1}{2}\phi_nw_1^{-1}\sum_{\alpha}g^{\alpha\alpha,nn}\xi_{\alpha}^2+\phi_nw_1^{-3}\sum_{\alpha,\beta}\kappa_{\alpha}\kappa_\beta\xi_{\alpha}^2\xi_{\beta}^2.
\end{align}
By \eqref{6.6} and \eqref{8.7} we obtain
\begin{align}
    \Big(\frac{\partial w_1}{\partial x_n}\frac{\partial b}{\partial x_n}\Big)(\phi,V)
    \cong -\phi_{nn}w_1^{-1}\sum_{\alpha}\kappa_{\alpha}\xi_{\alpha}^2.
\end{align}
According to \eqref{6.8} we compute
\begin{align}
    \Big(c_1\frac{\partial c_1}{\partial x_n}\Big)(\phi,V)
    \cong -\sum_{\alpha}\phi_{\alpha}\phi_{n\alpha}\xi_{\alpha}^2.
\end{align}
By \eqref{3.12.1} and \eqref{8.7} we get
\begin{align}
    &(bV)(\phi,V)
    \cong -(H+\phi)V,\\
    &\Big(V\frac{\partial w_1}{\partial x_n}\Big)(\phi,V)
    \cong w_1^{-1}V\sum_{\alpha}\kappa_{\alpha}\xi_{\alpha}^2.
\end{align}
It follows from \eqref{6.6} and \eqref{6.8} that 
\begin{align}
    &\Big(\sum_{\alpha,\beta}\frac{\partial w_1}{\partial \xi_\alpha}\frac{\partial w_1}{\partial \xi_\beta}\frac{\partial^2 b}{\partial x_{\alpha} \partial x_{\beta}}\Big)(\phi,V)
    \cong -w_1^{-2}\sum_{\alpha}\phi_{n\alpha\alpha}\xi_{\alpha}^2,\\
    &\Big(i\sum_{\alpha}\frac{\partial w_1}{\partial \xi_\alpha}\frac{\partial^2 c_1}{\partial x_{n} \partial x_{\alpha}}\Big)(\phi,V)
    \cong w_1^{-1}\sum_{\alpha}\phi_{n\alpha\alpha}\xi_{\alpha}^2,\\
    &\Big(i\sum_{\alpha}\frac{\partial c_1}{\partial \xi_\alpha}\frac{\partial b}{\partial x_{\alpha}}\Big)(\phi,V)
    \cong -\sum_\alpha \phi_\alpha\phi_{n\alpha}.
\end{align}

It follows from \eqref{8.7} that
\begin{align*}
    \frac{\partial^2 w_1}{\partial \xi_\alpha \partial x_n}(x_0)=2w_1^{-1}\kappa_\alpha\xi_\alpha-w_1^{-3}\xi_{\alpha}\sum_\beta \kappa_\beta\xi_\beta^2.
\end{align*}
Combining this and \eqref{a26} we get
\begin{align}
    \Big(i\sum_{\alpha}\frac{\partial^2 w_1}{\partial \xi_\alpha\partial x_n}\frac{\partial c_1}{\partial x_{\alpha}}\Big)(\phi,V)
    \cong 2w_1^{-1}\sum_\alpha\kappa_\alpha\phi_{\alpha\alpha}\xi_{\alpha}^2-w_1^{-3}\sum_{\alpha,\beta}\kappa_\alpha\phi_{\beta\beta}\xi_{\alpha}^2\xi_{\beta}^2.
\end{align}

Recall that $w_1=|\xi^\prime|$ by \eqref{2.6}. We compute
\begin{align*}
    \frac{\partial^2 w_1}{\partial \xi_\alpha\partial\xi_\beta}=w_1^{-1}g^{\alpha\beta}-w_1^{-3}\xi^\alpha\xi^\beta,\quad
    \frac{\partial^2 w_1}{\partial \xi_\alpha\partial\xi_\beta}(x_0)=w_1^{-1}\delta_{\alpha\beta}-w_1^{-3}\xi_\alpha\xi_\beta.
\end{align*}
Combining this and \eqref{6.6}, \eqref{7.1.2}, and using \eqref{7.1.5} and \eqref{7.1.1.1}, we obtain
\begin{align}
    &\Big(\sum_{\alpha,\beta}\frac{\partial^2 w_1}{\partial \xi_\alpha\partial \xi_\beta}\frac{\partial^2 b}{\partial x_{\alpha} \partial x_{\beta}}\Big)(\phi,V)
    \cong -w_1^{-1}\sum_{\alpha}\phi_{n\alpha\alpha}+w_1^{-3}\sum_{\alpha}\phi_{n\alpha\alpha}\xi_{\alpha}^2,\\
    &\frac{\partial^2 b}{\partial x_n^2}(\phi,V)
    \cong -\phi_{nnn},\\
    &\Big(b\sum_{\alpha,\beta}\frac{\partial^2 w_1}{\partial \xi_\alpha\partial \xi_\beta}\frac{\partial^2 w_1}{\partial x_{\alpha} \partial x_{\beta}}\Big)(\phi,V)
    \cong -\frac{1}{3}\phi_nw_1^{-2}\sum_{\alpha}R_{\alpha\alpha}\xi_\alpha^2.
\end{align}
Substituting the above results into \eqref{a27} we then get
\begin{align}
    &s_{-4}(\phi,V)
    \cong \frac{1}{8}(s_{-1}^4-s_{-1}^3w_1^{-1})(3H^2\phi_n+3H\phi_n^2+\phi_n^3) -\frac{1}{8}\Big((3s_{-1}^4w_1^{-2}-s_{-1}^3w_1^{-3} \notag\\
    &\quad -s_{-1}^2w_1^{-4})(2H\phi_n+\phi_n^2)+2(s_{-1}^3w_1^{-3}+s_{-1}^2w_1^{-4})\phi_{nn}\Big)\sum_\alpha\kappa_\alpha\xi_\alpha^2 +\frac{1}{8}(3s_{-1}^4w_1^{-4} \notag\\
    &\quad +5s_{-1}^3w_1^{-5})\phi_n\sum_{\alpha,\beta}\kappa_\alpha\kappa_\beta\xi_\alpha^2\xi_\beta^2 -\frac{1}{8}(3s_{-1}^4w_1^{-2}+s_{-1}^3w_1^{-3})(H+\phi_n)\sum_\alpha\phi_\alpha^2\xi_\alpha^2  \notag\\
    &\quad +\frac{1}{8}(3s_{-1}^4w_1^{-4}+5s_{-1}^3w_1^{-5}+4s_{-1}^2w_1^{-6})\sum_{\alpha,\beta}\kappa_\alpha\phi_\beta^2\xi_\alpha^2\xi_\beta^2 +\frac{1}{4}(3s_{-1}^4w_1^{-2}+s_{-1}^3w_1^{-3}) \notag\\
    &\quad \times \Big(\sum_\alpha\phi_\alpha\phi_{n\alpha}\xi_\alpha^2 +(H+\phi_n)\sum_\alpha\phi_{\alpha\alpha}\xi_\alpha^2  +\frac{1}{3}\phi_n\sum_\alpha R_{\alpha\alpha}\xi_\alpha^2\Big)  -\frac{1}{4}(3s_{-1}^4w_1^{-4}+5s_{-1}^3w_1^{-3} \notag\\
    &\quad +4s_{-1}^2w_1^{-6})\sum_{\alpha,\beta}\kappa_\alpha\phi_{\beta\beta}\xi_\alpha^2\xi_\beta^2 +\frac{1}{8}(2s_{-1}^3w_1^{-1}-s_{-1}^2w_1^{-2})H\phi_{nn} +\frac{1}{4}(2s_{-1}^3w_1^{-1}-s_{-1}^2w_1^{-2}) \notag\\
    &\quad \times \phi_n\sum_\alpha\kappa_\alpha^2 -\frac{1}{16}(2s_{-1}^3w_1^{-1}-s_{-1}^2w_1^{-2})\phi_n\sum_\alpha g_{\alpha\alpha,nn} +\frac{1}{8}(2s_{-1}^3w_1^{-1}-s_{-1}^2w_1^{-2})\phi_n\phi_{nn}  \notag\\
    &\quad -\frac{1}{8}s_{-1}^3w_1^{-3}\phi_n\sum_\alpha g^{\alpha\alpha,nn}\xi_\alpha^2 +\frac{1}{2}s_{-1}^3w_1^{-1}(H+\phi_n)V -\frac{1}{2}(s_{-1}^3w_1^{-3}+s_{-1}^2w_1^{-4})V\sum_\alpha\kappa_\alpha\xi_\alpha^2  \notag\\
    &\quad -\frac{1}{4}(2s_{-1}^4w_1^{-2}+s_{-1}^3w_1^{-3})\sum_\alpha\phi_{n\alpha\alpha}\xi_\alpha^2  -\frac{1}{8}(2s_{-1}^3w_1^{-1}+s_{-1}^2w_1^{-2})\sum_\alpha\phi_\alpha\phi_{n\alpha}  \notag\\
    &\quad +\frac{1}{2}(s_{-1}^3w_1^{3}+s_{-1}^2w_1^{-4})\sum_\alpha\kappa_\alpha\phi_{\alpha\alpha}\xi_\alpha^2  +\frac{1}{8}(2s_{-1}^3w_1^{-1}+s_{-1}^2w_1^{-2})\sum_\alpha\phi_{n\alpha\alpha}  \notag\\
    &\quad +\frac{1}{8}s_{-1}^2w_1^{-2}\phi_{nnn}+\frac{1}{4}s_{-1}^2w_1^{-2}\frac{\partial V}{\partial x_n}.
\end{align}

By applying formula \eqref{3.9}, Lemma \ref{lem4.1}, \ref{lem4.2}, and \ref{lem2.1}, we compute
\begin{align}
    a_3(\phi,V) & = \frac{\omega_{n-2}\Gamma(n-3)}{8(2\pi)^{n-1}}
        \biggl[
            \frac{n^4-9n^3+20n^2+7n-31}{2(n^2-1)}H^2\phi_n  + \frac{(n-3)(n^2-6n+6)}{2(n-1)}H\phi_n^2 \notag\\
            & \quad + \frac{(n-2)(n-3)}{6}\phi_n^3 + \frac{n^3-7n^2+9n+1}{2(n^2-1)}\phi_n\sum_{\alpha=1}^{n-1} \kappa_{\alpha}^2 + \sum_{\alpha=1}^{n-1} \phi_{\alpha}^2\kappa_{\alpha} - H\phi_{nn} \notag\\
            & \quad + \frac{n^2-6n+7}{2(n-1)}\tilde{R}\phi_n  - \frac{n^2-10n+15}{6(n-1)}R\phi_n + \frac{(n-2)(n-3)}{6(n-1)}\sum_{\alpha=1}^{n-1} \phi_{n\alpha\alpha} \notag\\
            & \quad + \Big(\frac{\partial }{\partial x_n} + (n-3)\phi_n + (n-4)H\Big) \Big(\Delta\phi - \frac{1}{2}|\nabla\phi|^2 + 2V\Big) 
        \biggr]. 
\end{align}
Therefore, combining this and \eqref{3.10}, \eqref{331}, we immediately obtain \eqref{a3}.

\begin{remark}
    Using formulas \eqref{2.9.1} and \eqref{2.14.1}, we can further calculate the lower order symbols $w_{-1-m}\ (m \geqslant 2)$ and $s_{-1-m}\ (m \geqslant 4)$. Therefore, by this procedure we can get all of the coefficients $a_k(x^\prime)$ for $0 \leqslant k \leqslant n-1$.
\end{remark}

\section{Proof of Corollary \ref{cor1.2}}\label{s7}

Since the sectional curvature of $M$ is a constant $K_0$. Then the Riemann curvature tensor of $M$ has the form (see \cite[p.\,183]{ChowKnopf04})
    \begin{align}\label{4.4}
        \tilde{R}_{jklm} = K_0(g_{jm}g_{kl} - g_{jl}g_{km}).
    \end{align}
    In this case, it follows from \eqref{ricci} and \eqref{scalar} that
    \begin{align}
        \label{9.13}
        &\tilde{R}_{jm} = K_0 \sum_{k,l}g^{kl}(g_{jm}g_{kl} - g_{jl}g_{km})= (n-1)K_0 g_{jm},\\
        \label{3.19.1}
        &\tilde{R} = \sum_{j,m}g^{jm} \tilde{R}_{jm}= n(n-1)K_0.
    \end{align}
    In boundary normal coordinates, from \eqref{4.4} and \eqref{9.13}, we have
    \begin{align}\label{3.22.3}
        \tilde{R}_{n \alpha\alpha n}(x_0)=K_0,\quad 
        \tilde{R}_{jk}(x_{0}) = (n-1)K_0 \delta_{jk}.
    \end{align}
    It follows from \eqref{4.1}, \eqref{4.2}, \eqref{3.19.1}, and \eqref{3.22.3} that, at $x_0$,
    \begin{align}\label{9.11}
        \sum_\alpha \kappa_\alpha^2(x_0)
        = (n-1)(n-2)K_0 + H^2 - R.
    \end{align}
    Moreover, at $x_0$,
    \begin{align}\label{9.12}
        \sum_\alpha \kappa_\alpha R_{\alpha\alpha}(x_0)
        &=\sum_\alpha \bigl[\kappa_\alpha (\tilde{R}_{\alpha\alpha} - \tilde{R}_{n \alpha\alpha n} + H\kappa_\alpha - \kappa_\alpha^2)\bigr] \notag\\
        &=\sum_\alpha \bigl[\kappa_\alpha ((n-1)K_0 - K_0 + H\kappa_\alpha - \kappa_\alpha^2)\bigr] \notag\\
        &=H^3 + n(n-2)HK_0 -HR -\sum_\alpha \kappa_\alpha^3. 
    \end{align}
    Substituting \eqref{3.19.1} and \eqref{9.11} into \eqref{a2} we then get \eqref{b2}. 
    
    It follows from formula \eqref{covariant} that
    \begin{align*}
        \nabla_{\frac{\partial}{\partial x_n}} \tilde{R}_{nn}
        = \frac{\partial \tilde{R}_{nn}}{\partial x_n} - 2\sum_j \Gamma^{j}_{nn} \tilde{R}_{jn}.\notag
    \end{align*}
    By \eqref{3.22.3}, we see that $\tilde{R}_{nn}(x_0)=(n-1)K_0$ is a constant, and by \eqref{chris}, we have
    \begin{align*}
        \Gamma^{k}_{nn}(x_0) = \frac{1}{2} \sum_j g^{jk}(2g_{nj,n} - g_{nn,j})
        = 0.
    \end{align*}
    Thus,
    \begin{align}\label{9.10}
        (\nabla_{\frac{\partial}{\partial x_n}}\tilde{R}_{nn})(x_0) =0.
    \end{align}
Substituting \eqref{3.19.1}, \eqref{9.11}, \eqref{3.22.3}, \eqref{9.10}, and \eqref{9.12}  into \eqref{a3} we then get \eqref{b3}. This completes the proof.

\section*{Acknowledgements}

This work was supported by the Postdoctoral Fellowship Program of CPSF under Grant Number GZC20240053 and the National Key R\&D Program of China under Grant Number 2020YFA0712800.


\begin{thebibliography}{99}
    \bibitem{Agra87} M. Agranovich, Some Asymptotic Formulas for Elliptic Pseudodifferential Operators, Funk. Anal. Prilozh. 21 (1987), 53--56.

    \bibitem{ANPS09} W. Arendt, R. Nittka, W. Peter, and F. Steiner, Weyl's Law: Spectral Properties of the Laplacian in Mathematics and Physics, Mathematical Analysis of Evolution, Information, and Complexity, Wiley-VCH Verlag GmbH \& Co. KGaA, Weinheim, 2009, pp. 1--72.

    \bibitem{Ara22} M. {Ara\'ujo Filho}, Estimates for the first eigenvalues of Bi-drifted Laplacian on smooth metric measure space, Differential Geometry and its Applications, 80 (2022), 101839.

    \bibitem{BranGilk90} T. Branson and P. Gilkey, The asymptotics of the Laplacian on a manifold with boundary, Comm. Partial Differential Equations 15 (1990), 245--272.

    \bibitem{Cald80} A. Calder\'{o}n, On an inverse boundary value problem, Seminar in Numerical Analysis and its Applications to Continuum Physics, Soc. Brasileira de Matemática, Rio de Janeiro, (1980), 65--73.

    \bibitem{Cek20} M. Ceki\'{c}, Calder\'{o}n problem for Yang--Mills connections, J. Spectr. Theory, 10 (2) (2020), 463--513.

    \bibitem{ChengZhou17} X. Cheng and D. Zhou, Eigenvalues of the drifted Laplacian on complete metric measure spaces, Commun. Contemp. Math. 19 (2017), no.~1, 1650001.

    \bibitem{ChowKnopf04} B. Chow and D. Knopf, The Ricci flow: An Introduction, Mathematical Surveys and Monographs, 110, Amer. Math. Soc., Providence, RI, 2004.

    \bibitem{ChowLL06} B. Chow, P. Lu, and L. Ni, Hamilton's Ricci Flow, Graduate Studies in Mathematics, 77, Amer. Math. Soc., Providence, RI, 2006 Sci. Press Beijing, New York, 2006.

    \bibitem{DuisGuill75} H. Duistermaat and V. Guillemin, Spectrum of positive elliptic operators and periodic bicharacteristics, Invent. Math. 29(1) (1975), 39--79.

    \bibitem{Edward91} J. Edward and S. Wu, Determinant of the Neumann operator on smooth Jordan curves, Proc. Amer. Math. Soc. 111(2) (1991), 357--363.

    \bibitem{FoxKuttler83} D. Fox and J. Kuttler, Sloshing frequencies, Z. Angew. Math. Phys. 34 (1983), 668--696.

    \bibitem{Fried64} A. Friedman, Partial Differential Equations of Parabolic Type, Prentice Hall, Englewood Cliff, 1964.

    \bibitem{Full94} S. Fulling (Ed.), Heat Kernel Techniques and Quantum Gravity, Discourses in Mathematics and its Applications, 4, Texas A \& M University, Department of Mathematics, College Station, TX, 1995.

    \bibitem{FutaLL13} A. Futaki, H. Li, and X. Li, On the first eigenvalue of the Witten-Laplacian and the diameter of compact shrinking solitons, Ann. Global Anal. Geom. {\bf 44} (2013), no.~2, 105--114.

    \bibitem{Gilkey04} P. Gilkey, Asymptotic Formulae in Spectral Geometry, CRC Press, Boca Raton, 2004.

    \bibitem{Gilkey95} P. Gilkey, Invariance Theory, the Heat Equation and the Atiyah--Singer Index Theorem, Second Edition, CRC Press, Boca Raton, 1995.

    \bibitem{Gilkey75} P. Gilkey, The spectral geometry of a Riemannian manifold, J. Differential Geom. 10 (1975), 601--618.

    \bibitem{GilkeyGrubb98} P. Gilkey and G. Grubb, Logarithmic terms in asymptotic expansions of heat operator traces, Comm. Partial Differential Equations {\bf 23} (1998), no.~5-6, 777--792.

    \bibitem{Grubb86} G. Grubb, Functional Calculus of Pseudo-Differential Boundary Problems, Birkh\"{a}user, Boston, 1986.

    \bibitem{GrubbSeeley95} G. Grubb and R. Seeley, Weakly parametric pseudodifferential operators and Atiyah--Patodi--Singer boundary problems, Invent. Math. 121 (1995), 481--529.

    \bibitem{Hormander64} L. H\"{o}rmander, Linear Partial Differential Operators, Springer, Berlin-New York, 1976.

    \bibitem{Hormander85.3} L. H\"{o}rmander, The Analysis of Partial Differential Operators III. Pseudo-differential operators. Reprint of the 1994 edition, Classics in Mathematics, Springer, Berlin, 2007.

    \bibitem{Kac66} M. Kac, Can one hear the shape of a drum?, Amer. Math. Monthly 73(4) (1966), 1--23.

    \bibitem{KopaKrein01} N. Kopachevsky and S. Krein, Operator Approach to Linear Problems of Hydrodynamics, vol. 1, Oper. Theory Adv. Appl., vol. 128, Birkhäuser-Verlag, Basel, 2001.

    \bibitem{Kuch01} P. Kuchment, The mathematics of photonic crystals, Mathematical modeling in optical science. Frontiers in Applied Mathematics, 22. SIAM, Philadelphia, PA, 2001, 207--272.

    \bibitem{LeeUhlm89} J. M. Lee and G. Uhlmann, Determining anisotropic real-analytic conductivities by boundary measurements, Commun. Pure Appl. Math. 42(8) (1989), 1097--1112.

    \bibitem{Liu22S} G. Liu, The geometric invariants for the spectrum of the Stokes operator, Math. Ann. 382(3-4) (2022), 1985--2032.

    \bibitem{Liu11} G. Liu, The Weyl-type asymptotic formula for biharmonic Steklov eigenvalues on Riemannian manifolds, Adv. Math. 228(4) (2011), 2162--2217.

    \bibitem{Liu15} G. Liu, Asymptotic expansion of the trace of the heat kernel associated to the Dirichlet-to-Neumann operator, J. Differential Equations 259(7) (2015), 2499--2545.

    \bibitem{Liu22p} G. Liu, Spectral invariants of the perturbed polyharmonic Steklov problem, Calc. Var. Partial Differential Equations {\bf 61} (2022), no.~4, Paper No. 125, 19 pp.

    \bibitem{Liu21} G. Liu, Geometric Invariants of Spectrum of the Navier--Lam\'{e} Operator, J. Geom. Anal. 31(2021), 10164--10193.

    \bibitem{Liu19} G. Liu, Determination of isometric real-analytic metric and spectral invariants for elastic Dirichlet-to-Neumann map on Riemannian manifolds, \href{https://arxiv.org/abs/1908.05096}{arXiv:1908.05096}, 2019.

    \bibitem{LiuTan23} G. Liu and X. Tan, Spectral invariants of the magnetic Dirichlet-to-Neumann map on Riemannian manifolds, J. Math. Phys. 64 (2023), 041501.

    \bibitem{LiuTan22.1} G. Liu and X. Tan, Asymptotic expansions of the traces of the thermoelastic operators, \href{https://arxiv.org/abs/2205.13238}{arXiv:2205.13238}, 2022.

    \bibitem{LiuTan22.2} G. Liu and X. Tan, Asymptotic expansion of the heat trace of the thermoelastic Dirichlet-to-Neumann map, \href{https://arxiv.org/abs/2206.01374}{arXiv:2206.01374}, 2022.

    \bibitem{Loren47} G. Lorentz, Beweis des Gaussschen Integralsatzes, Math. Z. 51 (1947), 61--81.

    \bibitem{LuRow12} Z. Lu and J.~M. Rowlett, Eigenvalues of collapsing domains and drift Laplacians, Math. Res. Lett. {\bf 19} (2012), no.~3, 627--648.

    \bibitem{Morrey66} C. Morrey, Jr., Multiple Integrals in the Calculus of Variations, Springer-Verlag, New York, Inc., 1966.

    \bibitem{Morrey58} C. Morrey, Jr., On the analyticity of the solutions of analytic non-linear elliptic systems of partial differential equations, Amer. J. Math. 80 (1958), 198--218.

    \bibitem{PoltSher15} I. Polterovich and D. Sher, Heat Invariants of the Steklov Problem, J. Geom. Anal. 25 (2015), 924--950.

    \bibitem{Protter87} M. Protter, Can One Hear the Shape of a Drum? Revisited, SIAM Rev. 29(2) (1987), 185--197.

    \bibitem{SafarovVass97} Yu. Safarov and D. Vassiliev, The Asymptotic Distribution of Eigenvalues of Partial Differential Operators, American Mathematical Society, 1997.

    \bibitem{Sandgren55} L. Sandgren, A Vibration Problem, Medd. Lunds Univ. Mat. Sem. 13 (1955), 1--83.

    \bibitem{Seeley67} R. Seeley, Complex Powers of an Elliptic Operator, In: Singular Integrals, Proc. Symp. Pure Math., Amer. Math. Soc., Province, RI., (1967), pp.\,288--307.

    \bibitem{Seeley69} R. Seeley, The resolvent of an elliptic boundary value problem, Amer. J. Math. 91 (1969), 889--920.

    \bibitem{Shubin01} M. A. Shubin, Pseudodifferential Operators and Spectral Theory, Second Edition, Springer-Verlag, Berlin Heidelberg New York, 2001.

    \bibitem{Stekloff02} W. Stekloff, Sur les probl\`{e}mes fondamentaux de la physique math\'{e}matique (suite et fin), Ann. Sci. \'{E}cole Norm. Sup. 19 (1902), 455--490.

    \bibitem{Stewart74} H. Stewart, Generation of analytic semigroups by strongly elliptic operators, Trans. Amer. Math. Soc. 199 (1974), 141--161.

    \bibitem{SylvUhlm87} J. Sylvester and G. Uhlmann, A global uniqueness theorem for an inverse boundary value problem, Ann. Math. 125(1) (1987), 153--169.

    \bibitem{Taylor11.2} M. Taylor, Partial Differential Equations II. Qualitative studies of linear equations, Second edition, Applied Mathematical Sciences, 116, Springer, Cham, 2011.

    \bibitem{Taylor81} M. Taylor, Pseudodifferential Operators, Princeton University Press, Princeton, New Jersey, 1981.

    \bibitem{Treves80} F. Treves, Introduction to pseudodifferential and Fourier integral operator, Vol. 1, Pseudodifferential operators, Univ. Ser. Math., Plenum, New York-London, 1980.

    \bibitem{Uhlmann14} G. Uhlmann, Inverse problems: seeing the unseen, Bull. Math. Sci. 4(2) (2014), 209--279.

    \bibitem{WangWang19} W. Wang and Z. Wang, On the relative heat invariants of the Dirichlet-to-Neumann operators associated with Schrödinger operators, J. Pseudo-Differ. Oper. Appl. 10 (2019), 805--836.

    \bibitem{Weyl12} H. Weyl, \"Uber die Abh\"angigkeit der Eigenschwingungen einer Membran und deren Begrenzung, J. Reine Angew. Math. 141 (1912), 1--11.
\end{thebibliography}
\end{document}